\documentclass[notitlepage]{amsart}

\newtheorem{pro}{Proposition}[section]
\newtheorem{teo}[pro]{Theorem}
\newtheorem{defi}[pro]{Definition}
\newtheorem{lem}[pro]{Lemma}
\newtheorem{cor}[pro]{Corollary}
\newtheorem{rk}[pro]{Remark}
\newtheorem{ex}[pro]{Example}

\newcommand{\Ext}{\mathrm{Ext}}
\newcommand{\Tor}{\mathrm{Tor}}

\newcommand{\Hom}{\mathrm{Hom}}

\newcommand{\A}{\mathcal{A}}
\newcommand{\B}{\mathcal{B}}

\newcommand{\I}{\mathcal{I}}

\newcommand{\C}{\mathcal{C}}
\newcommand{\F}{\mathcal{F}}

\newcommand{\Le}{\mathcal{L}}

\newcommand{\X}{\mathcal{X}}
\newcommand{\Y}{\mathcal{Y}}
\newcommand{\Z}{\mathcal{Z}}
\newcommand{\W}{\mathcal{W}}
\newcommand{\pd}{\mathrm{pd}}
\newcommand{\op}{\mathrm{op}}

\newcommand{\Gpd}{\mathrm{Gpd}}

\newcommand{\Proj}{\mathcal{P}}

\newcommand{\Inj}{\mathcal{I}}

\newcommand{\id}{\mathrm{id}}
\newcommand{\Gid}{\mathrm{Gid}}
\newcommand{\resdim}{\mathrm{resdim}}

\newcommand{\coresdim}{\mathrm{coresdim}}

\newcommand{\Modu}{\mathrm{Mod}}
\newcommand{\Ker}{\mathrm{Ker}}

\newcommand{\GP}{\mathcal{GP}}
\newcommand{\DP}{\mathcal{DP}}
\newcommand{\DI}{\mathcal{DI}}
\newcommand{\GF}{\mathcal{GF}}
\newcommand{\Flat}{\mathcal{F}}

\newcommand{\fd}{\mathrm{fd}}
\newcommand{\Gfd}{\mathrm{Gfd}}

\newcommand{\GI}{\mathcal{GI}}

\newcommand{\Ch}{\mathrm{Ch}}

\usepackage{latexsym,amssymb,amscd} 
\usepackage{amsmath}
\usepackage[all]{xypic}
\usepackage{amsthm}

\newcommand{\cogorro}{\vee}
\newcommand{\ortogonal}{\bot}
\newcommand{\gorro}{\wedge}

\begin{document}
\title[$(\Flat, \A)$-Gorenstein flat homological dimensions]{$(\Flat, \A)$-Gorenstein flat homological dimensions}
\author{V\'ictor Becerril}
\address[V. Becerril]{Centro de Ciencias Matem\'aticas. Universidad Nacional Aut\'onoma de M\'exico. 
 CP58089. Morelia, Michoac\'an, M\'EXICO}
\email{victorbecerril@matmor.unam.mx}

\thanks{Key Words: }
\begin{abstract} 
 In this paper we develop the homological properties of the $(\Le, \A)$-Gorenstein flat $R$-modules $\GF_{(\Flat (R), \A)}$ proposed by Gillespie. Where the class $\A \subseteq \Modu (R^{\op})$ sometimes corresponds to a duality pair $(\Le, \A)$. We study the weak global and finitistic dimensions that comes with $\GF_{(\Flat (R), \A)}$ and show that over a $(\Le, \A)$-Gorenstein ring, the functor $-\otimes _R-$ is left balanced over $\Modu (R^{\op}) \times \Modu (R)$ by the classes $\GF_{(\Flat (R^{\op}), \A)} \times \GF_{(\Flat (R), \A)}$. When the duality pair is $(\Flat (R), \mathcal{FP}Inj (R^{\op}))$ we recover the G. Yang's result over a Ding-Chen ring, and we see that is new for $(\mathrm{Lev} (R), \mathrm{AC} (R^{\op}))$ among others.
\end{abstract}  
\maketitle

\section{Introduction} 

Duality pairs were defined and studied by H. Holm and P. J\o rgensen \cite{HJ}. Recall that for a left $R$-module $M$, its character module is defined to be the right $R$-module $M ^{+} := \Hom _{R} (M, \mathbb{Q}/ \mathbb{Z})$. A pair of classes $(\Le, \A) \subseteq \Modu (R) \times \Modu (R^{\op})$ is a duality pair if is such that  $L \in \Le $ if and only if $L ^{+} \in \A$, and $\A$ is closed under direct summands and finite direct sums.
 J. Gillespie \cite{Gill18} define the notion of $\mathrm{AC}$-Gorenstein ring which is a generalization of a Gorenstein ring compatible with the $\mathrm{AC}$-Gorenstein projective $R$-modules \cite[Definition 2.3]{BGH}. In \cite{Gill18} is described certain model structures that comes from this $\mathrm{AC}$-Gorenstein projective $R$-modules. In other paper J. Gillespie has introduced the notion of $(\Le, \A)$-Gorenstein projective, injective and flat $R$-modules  \cite{Gill19} (resp. denoted in this paper by $\GP_{(\Proj , \Le)}$, $\GI _{(\A, \Inj)}$ and $\GF_{(\Flat, \A)}$) that comes from a duality pair $(\Le, \A)$. Based on these ideas J. Wang and Z. Di \cite{Wang20} defines the concept of $(\Le, \A)$-Gorenstein ring with respect to a \textit{bi-complete} duality pair showing that such notion unify the notions of Gorenstein, Ding-Chen, $\mathrm{AC}$-Gorenstein, and Gorenstein $n$-coherent ring taking an appropriate duality pair. Among others results J. Wang and Z. Di  \cite[Theorem 4.8]{Wang20} have been proven that   over a \textit{bi-complete} duality pair $(\Le, \A)$ there is an hereditary and complete cotorsion triple that comes from the $(\Le, \A)$-Gorenstein projective, and $(\Le, \A)$-Gorenstein injective $R$-modules. This result recovers the G. Yang's result \cite[Theorem 3.6]{Yang} when $R$ is Ding-Chen and for the duality pair $(\Flat (R), \mathcal{FP}Inj (R^{\op}))$. More recently J. Gillespe and A. Iacob \cite[Corollary 5.1]{Gill21} extends the study of model structures given in \cite{BGH}, and \cite{Gill18} of a general way for the classes of $R$-modules $(\Le, \A)$-Gorenstein projective, injective and flat introduced in \cite{Gill19}. Thus, in the current literature, there is a treatment of the homotopic properties of the classes  $\GP_{(\Proj , \Le)}$, $\GI _{(\A, \Inj)}$ and $\GF_{(\Flat, \A)}$.
 
  We therefore see the importance to develop their homologic properties, which is the aim of the present paper. We develop this properties of the following manner.  In Section \ref{S2} we present general concepts from relative homological algebra in terms of an abelian category $\C$ which are important in the study of balance, such as relative homological dimensions, left and right approximations, relative cogenerators, cotorsion pairs.  Also recall the notion of GP-admissible pair and see how  some duality pairs and their properties give us GP-admissible pairs. Section \ref{S3} is devoted to develop relative homological dimensions that comes from the class $\GF _{(\Flat, \A)}$. In Section \ref{S4} we study the natural finitistic  and weak global dimensions relative to the class $\GF _{(\Flat, \A)}$,  we prove in Theorem \ref{EquivalenciaGlobal} that the left weak global dimension that comes from $\GF _{(\Flat, \A)}$ is characterized by the flat dimension of the class $\A$, we also prove in Proposition \ref{UnderLimit} that under certain conditions there is an under limit  given by the flat dimension of the class $\Le$. We also compare in Corollary \ref{GlobalFinitista} this  left weak global dimension with the finitistic dimension obtained from  $\GF _{(\Flat, \A)}$. The main results are presented in Section \ref{S5}. We first discussed the usefulness of the bi-complete pairs to define a $(\Le, \A)$-Gorenstein ring in the sense of Wang \cite[\S 4]{Wang20}. Also we see how the notion of GP-admissible pair \cite{BMS} enables us to obtain a relative $\Hom$-balance result for the classes $\GP_{(\Proj , \Le)}$, $\GI _{(\A, \Inj)}$ in Lemma \ref{LemaBalance}. Furthermore we prove in Lemma \ref{BalanBiperfect} that for a $(\Le, \A)$-Gorenstein ring the functor $-\otimes _R-$ is left balanced over the whole category $\Modu (R)$ by the class  $\GF _{(\Flat, \A)}$. Finally in Section \ref{S6} we give examples of the relative $\Hom$-balance and several examples of $(\Le, \A)$-Gorenstein rings where the $-\otimes_R -$ occurs and we see how one of this examples recover the well-know result on a Ding-Chen ring \cite[Theorem 3.23]{Yang}. Finally in Corollary \ref{Finitista} and Proposition \ref{IgualdadDing} we answer a question proposed by A. Iacob \cite{Alina20}, namely; \textit{When it is true that the Ding projective modules and the Gorenstein projective modules coincide?}. We also give conditions in Proposition \ref{Aplastado}  for other Gorenstein classes to coincide with the Gorenstein projective $R$-modules.

\section{Preliminaries} \label{S2}
In this section we will declare some of notation in terms of an abelian category $\C$. The general notation presented here will enable us to give short proofs and more clear concepts in the following sections. We denote by $\pd (C)$ the \textbf{projective dimension} of $C \in \C$, and by $\Proj(\C)$ the class of all the objects  $C \in \C$ with $\pd (C)=0$. Similarly, $\id (C) $ stands for the \textbf{injective dimension} of $C \in \C$, and $\Inj(\C)$ for the class of all the objects $C \in \C$ with $\id (C) =0$. Monomorphism and epimorphism in $\C$ may sometimes be denoted using arrows $\hookrightarrow$ and $\twoheadrightarrow$, respectively.

 Let $\X $ be  a class of objects in $\C$ and $M \in \C$. We set the following notation:
 \textit{Orthogonal classes}.  For each positive integer $i$, we consider the right orthogonal classes 
$$\X^{\ortogonal _i} := \{N \in \C: \Ext ^{i} _{\C} (-,N) | _{\X} =0\} \mbox{ and } \X ^{\ortogonal} := \cap _{i >0} \X^{\ortogonal _i}.$$
Dually, we have the left orthogonal classes $^{\ortogonal _i} \X$ and $^{\ortogonal} \X$.

Given a class $\Y\subseteq \C$, we write $\X \ortogonal \Y$ whenever $\Ext ^{i} _{\C} (X, Y) =0$ for all $X \in \Y$,  $Y \in \Y$ and $i >0$.  

 \textit{Relative homological dimensions}. The \textbf{relative projective dimension of $M$, with respect to $\X$}, is defined as
$$\pd _{\X} (M) : = \min \{n \in \mathbb{N}: \Ext ^{j} _{\C} (M,-) | _{\X} =0 \mbox{ for all } j>n\}.$$
We set $\min \emptyset := \infty$. Dually, we denote by $\id _{\X} (M)$ the \textbf{relative injective dimension of $M$ with respect to $\X$}. Furthermore,  we set  
$$\pd _{\X} (\Y) := \sup\{ \pd _{\X} (Y): Y \in \Y\} \mbox{ and } \id _{\X} (\Y) := \sup \{\id _{\X} (Y): Y \in \Y\}. $$
If $\X = \C$, we just write $\pd (\Y)$ and $\id (\Y)$.

 \textit{Resolution and coresolution dimension}. The \textbf{$\X$-coresolution dimension of $M$} denoted $\coresdim _{\X} (M)$, is the smallest non-negative integer $n$ such that there is an exact sequence 
$$0 \to M \to X_0 \to X_1 \to \cdots \to X_n \to 0,$$ 
with $X_i \in \X$ for all $i \in \{0, \dots , n\}$. If such $n$ does not exist, we set $\coresdim _{\X} (M) := \infty$. Also, we denote by $\X^{\cogorro } _n$ the class of objects in $\C$ with $\X$-coresolution dimension at most $n$. The union $\X ^{\cogorro} := \cup _{n \geq 0} \X ^{\cogorro} _n$ is the class of objects in $\C$ with finite $\X$-coresolution dimension.
Dually, we have the $\X$-{\bf{resolution dimension}} $\resdim_\X\,(M)$ of $M,$  $\mathcal{X}^{\wedge} _n$  the class of objects in 
$\C$ having  $\mathcal{X}$-resolution dimension at most $n$, and the union $\X ^{\gorro} := \cup _{n \geq 0} \X ^{\gorro} _n$ is the class of objects in $\C$ with finite $\X$-resolution dimension. We set  
$$\coresdim_\X\,(\Y):=\mathrm{sup}\,\{\coresdim_\X\,(Y)\;
:\; Y\in\Y\},$$ and  $\resdim_\X\,(\Y)$ is defined dually.

 \textit{Proper resolutions and balance.}
Given a class $\X \subseteq  \C$ a  \textbf{left proper $\X$-resolution} of $M \in \C $ is a complex $\mathbf{X} (M):\cdots \to X^1 \to X ^0 \to M \to 0$  such that the complex $\Hom _{\C} (X,\mathbf{X})$ is acyclic for all $X \in \X$. A \textbf{right proper $\X$-coresolution} is defined dually. We recall that if $\X, \Y, \mathcal{B} , \mathcal{E}$ are classes of objects in $\C$, then we say that $\Hom _{\C} (-,-)$ is \textbf{right balanced} on $\X \times \Y$ by $\mathcal{B} \times \mathcal{E}$ if for any objects $X \in \X$ and $\Y \in \Y$ there exist complexes 
$$\cdots \to B_2 \to B_1 \to B_0 \to X \to 0$$ 
and
$$0 \to Y \to E^0 \to E^1 \to E^2 \to \cdots $$
such that $B_i \in \mathcal{B}$, $E^{i} \in \mathcal{E}$ for all $i \geq 0$ and such that $\Hom _{\C} (-, E)$ makes the first complex acyclic whenever $E \in \mathcal{E}$ and such that $\Hom _{\C} (B, -)$ makes the second complex acyclic whenever  $B \in \mathcal{B}$. In the case of the category $\Modu (R)$ of left $R$-modules over an associative ring $R$ the definition is easily modified to give the definition  of a left or a right balanced functor $- \otimes _R -$ (see \cite{EnJen92,EnJen95,EnJen00} for details).

The class $\X$ is called \textbf{precovering} if for each $M \in \C$ there is  a homomorphism $f : X \to M$ such that $\Hom _{\C} (Z, f) : \Hom _{\C} (Z, X) \to \Hom _{\C} (Z,M)$ is surjective for any $Z \in \X$. Dually is defined a \textbf{preenveloping} class.

 \textit{Relative Gorenstein objects}. Given a pair $(\X,\Y)$ of classes of objects in $\C$, an object $M \in \C$ is:
 \textbf{$(\X,\Y)$-Gorenstein projective} \cite[Definition 3.2]{BMS} if $M$ is a cycle of an exact complex $\mathbf{X}$ with $X_m \in \X$ for every $m \in \mathbb{Z}$, such that the complex $\Hom_{\C}(\mathbf{X},Y)$ is an exact complex for all $Y \in \Y$. The class of all $(\X, \Y)$-Gorenstein projective objects is denoted by $\GP _{(\X, \Y)}$.  

 The  \textbf{$(\X,\Y)$-Gorenstein  projective dimension of $M$} is defined by 
$$\Gpd_{(\X,\Y)}(M):=\resdim_{\GP_{(\X,\Y)}}(M),$$
and for  any class $\mathcal{Z}\subseteq\A,$   $\Gpd_{(\X,\Y)}(\mathcal{Z}):=\sup \{\Gpd_{(\X,\Y)}(Z)\;:\;Z\in\mathcal{Z}\}.$   

Dually, we have the notion of $(\X, \Y)$-Gorenstein injective objects and their dimensions.

Let $(\X, \omega) \subseteq \C ^2$. The class $\omega$ is \textbf{$\X$-injective} if $\id _{\X}(\omega) =0$. Is said that $\omega$ is a \textbf{relative cogenerator} in $\X$ if $\omega \subseteq \X$ and  for any $X \in \X$ there is an exact sequence $0 \to X \to W \to X' \to 0$, with $W \in \omega$ and $X' \in \X$.  Dually, we have the notions of \textbf{$\X$-projective} and \textbf{relative generator} in $\X$.
We recall another notion from \cite{BMS}.

\begin{defi}\cite[Definition 3.1]{BMS}
A pair $(\X, \Y) \subseteq \C$ is \textbf{GP-admissible} if for each $C \in \C$, there is an epimorphism $X \to C$, with $X \in \C$, and $(\X, \Y)$ satisfies the following conditions:
\begin{itemize}
\item[(a)] $\X$ and $\Y$ are closed under finite coproducts in $\C$, and $\X$ is closed under extensions;
\item[(b)]   $\omega := \X \cap \Y$ is a relative cogenerator in $\X$ and $\X \ortogonal \Y$.
\end{itemize}
\end{defi}

GI-admissible pairs are defined dually (see \cite[Definition 3.6]{BMS}).\\ 

A pair $(\X, \Y) \subseteq \C^{2} $ is a \textbf{cotorsion pair} if $\X^{\ortogonal _1} = \Y$ and $\X = {^{\ortogonal _1} \Y}$. This cotorsion pair is  \textbf{complete} if for any $C \in \C$, there are exact sequences $0 \to Y \to X \to C \to 0$ and $0 \to C \to Y' \to X' \to 0$ where $X, X' \in \X $ and $Y, Y' \in \Y$. Moreover, the cotorsion pair is \textbf{hereditary} if $\X \ortogonal \Y$. A triple $(\X,\Y,\Z) \subseteq \C ^{3}$ is called a \textbf{cotorsion triple} \cite{Chen} provided that both $(\X, \Y)$ and $(\Y, \Z)$ are cotorsion pairs; it is \textbf{complete} (resp. \textbf{hereditary}) provided that both of the two cotorsion  pairs are complete (resp. hereditary).\\
Now we turn our attention to the category $\Modu (R)$ of left $R$-modules, where $R$ is an associative ring. To refer to an element in $\Modu (R)$ we will say simply \textit{$R$-module}, while for a right $R$-module we will say \textit{$R^{\op}$-module} and we denote this last category by $\Modu (R^{\op})$. For short we will write $\Proj (R)$ and $\Inj (R)$ for the classes $\Proj (\Modu (R))$ and $\Inj (\Modu (R))$ respectively, and $\Flat (R)$ for the class of flat $R$-modules.

Recall that for a given $ R$-module $M$, its \textit{character module} is defined to be the $R ^{\op}$-module $M ^{+} := \Hom _{\mathbb{Z}} (M, \mathbb{Q/Z})$. 

\begin{defi}\cite[Definition 2.1]{HJ}
A \textbf{duality pair} over a ring $R$ is a pair $(\Le, \A)$, where $\Le$ is a class of $R$-modules and $\A$ is a class of $R^{\op}$-modules, satisfying the following conditions:
\begin{itemize}
\item[(1)] $M \in \Le$ if and only if $M ^{+} \in \A$.
\item[(2)] $\A$ is closed under direct summands and finite direct sums.
\end{itemize}
A duality pair $(\Le, \A)$ is called \textbf{perfect} if $\Le$ contains the $R$-module $_R R$ and is closed under coproducts and extensions.
\end{defi}
By the Lambek's Theorem, the more natural example of a duality pair is when we consider the class $\F (R)$ of all flat $R$-modules and the class $\I (R^{\op})$ of all injective $R^{\op}$-modules. We are interested when $(\Proj (R) , \Le)$ give us a GP-admissible pair and thus, it will be possible to obtain a similar result as the one in \cite[Lemma 4.6]{Wang20} (by using \cite[Proposition 3.16]{BMS}), among others but with less hypothesis.  The following result will be useful to see how to do this. 

\begin{pro}\cite[Proposition 2.3]{Gill19} \label{Perfect01}
If $(\Le, \A)$ is a perfect duality pair, then $\Le$ contains all projective $R$-modules and $\A$ contains all injective $R^{\op}$-modules. In fact $\Le$ is closed under direct limits and so contains all flat $R$-modules too.
 \end{pro}



The following notion comes originally by Gillespie \cite{Gill19} but it was only stated there for commutative rings. Recently has been extended to non-commutative rings in \cite{Gill21}, but earlier in \cite[Definition 3.1 (4)]{Wang20}.
 
 \begin{defi}\cite[Definition 2.3]{Gill21}
 By a \textbf{symmetric} duality pair $\{\Le, \A \}$ we mean:
 \begin{itemize}
 \item[(1)] $\Le$ is a class of $R$-modules.
 \item[(2)] $\A$ is a class of $R^{\op}$-modules.
 \item[(3)] $(\Le, \A)$ is a duality pair over $R$ and $(\A, \Le)$ is a duality pair over $R^{\op}$.
 \end{itemize}
 \end{defi}
 
 An example of a symmetric duality pair is obtained when $R$ is a noetherian ring by taking $\{ \F (R), \Inj (R^{\op}) \}$ (see \cite[Example 3.7]{Wang20}).
 
  Another example is obtained by taking $\Le$ to be the class of all \textbf{level} $R$-modules and $\A$ to be the class of all \textbf{absolute clean} $R^{\op}$-modules \cite{BGH}. In the last section we will see in detail such classes of $R$-modules.  We denote by $\mathrm{AC} (R)$  the class  of all absolutely clean $R$-modules, and by  $\mathrm{Lev} (R)$ the class of all level $R$-modules.  J. Gillespie and A. Iacob  call \textbf{semi-perfect}  to a duality pair $(\Le, \A)$ if it has all the properties required to be a perfect duality pair \textit{except} that $\Le$ may not be closed under extensions \cite{Gill21}. With this in mind we have the following.

\begin{defi}\cite[Definition 2.5]{Gill21}
By a \textbf{semi-complete duality pair} $(\Le , \A)$ we mean that $\{ \Le , \A\}$ is a symmetric duality pair with $(\Le, \A)$ being a semi-perfect duality pair. If $(\Le, \A)$ is indeed perfect, then we call it \textbf{complete duality pair}.
\end{defi}

By \cite[Remark 2.6]{Gill21} we see that when $(\Le, \A)$ is a semi-complete duality pair  then $\Le$ contains all projective and flat  $R$-modules and $\A$ contains all injective $R^{\op}$-modules. 
\section{Modules of $(\Le, \A)$-Gorenstein type}\label{S3} 

Recently Gillespie considers in \cite{Gill19} classes of  Gorenstein projective, Gorenstein flat  $R$-modules and Gorenstein injective $R^{\op}$-modules, relative to a duality pair $(\Le, \A)$ as follows. Let $M$ be a $R$-module and $N$ be a $R^{\op}$-module.
\begin{itemize}
\item $M$ is \textbf{Gorenstein $(\Le, \A)$-projective} if $M = Z_0 (\mathbf{P})$ for some exact complex of  projective  $R$-modules $\mathbf{P}$ for which $\Hom _{R} (\mathbf{P}, L)$ is acyclic for all $L \in \Le$.
\item $M$ is \textbf{Gorenstein $(\Le, \A)$-flat} if $M = Z_0 (\mathbf{F})$ for some exact complex of flat $R$-modules for which $A \otimes _R \mathbf{F}$ is acyclic for all $A \in \A$.
\item $N$ is \textbf{Gorenstein $(\Le, \A)$-injective} if $N = Z_0 (\mathbf{I})$ for some exact complex of injective $R^{\op}$-modules $\mathbf{
I}$ for which $\Hom _{R} (A,\mathbf{I})$ is acyclic for all $A \in \A$.
\end{itemize}

We denote the previous $R$-modules classes by  $\GP _{(\Proj, \Le)}$, $\GF _{(\Flat, \A)}$, and the $R^{\op}$-module class $\GI _{(\A, \Inj (R^{\op}))}$  respectively.
It can be easily seen that when $(\Le, \A)$ is a perfect  or semi-complete duality pair, then $(\Proj (R) , \Le)$ is a GP-admissible pair in $\Modu (R)$, and $(\A, \Inj(R^{\op}))$ is GI-admissible pair in $\Modu (R^{\op})$. From this we can obtain a similar result that the one in \cite[Lemma 4.6]{Wang20} by applying the theory developed in \cite{BMS}, but for a  perfect  or semi-complete duality pair instead of a \textit{bi-complete} duality pair (see Definition \ref{Bi-complete}).

 Since $(\mathrm{Lev}(R), \mathrm{AC}(R^{\op}))$ is a complete duality pair \cite[Example 2.7]{Gill21} then the pair $(\Proj (R), \mathrm{Lev} (R))$ is GP-admissible  in $\Modu (R)$ and  $(\mathrm{AC}(R^{\op}), \Inj (R^{\op}))$ is GI-admissible in $\Modu (R^{\op})$. The  $R$-module classes  $\GP _{(\Proj, \mathrm{Lev} (R))}$ and $\GI_{(\mathrm{AC}(R), \Inj)}$ are called \textbf{Gorenstein AC-projective} and \textbf{Gorenstein AC-injective} \cite{Gill18}, respectively. Note that we can consider the classes $\mathrm{AC}(R)$ and $\mathrm{AC}(R^{\op})$ which are defined in $\Modu (R) $ and in $\Modu (R^{\op})$, respectively, we will see later how this is useful to us. There is another relationship with Gorenstein categories which is set out below.

\begin{rk}
From  \cite[Theorem 2.5]{Gill19} when $(\Le, \A)$ is a symmetric duality pair in $\Modu (R)$ and $R$ is a commutative ring, the  the class $\GP _{(\Proj, \Le)}$ of all Gorenstein $(\Le, \A)$-projectives is precisely  the subcategory $(\Proj, \Le, \A)$-Gorenstein presented in  \cite[Definition 2.1]{Yang15}.  Thus, some results in this paper  will be applicable to this Gorenstein categories.
\end{rk}

\begin{lem}
Consider $R$ a right coherent ring, an exact complex $I^{\bullet}$  of injective $R^{\op}$-modules and suppose that $\fd (A) < \infty$ for all $A \in \A \subseteq \Modu (R^{\op})$. Then the exact complex of flat $R$-modules  $I^{\bullet +}$ is $(\A \otimes _R -)$-acyclic. 
\end{lem}

\begin{proof}
Let us consider $I^{\bullet}$ the exact  complex of injective  $R^{\op}$-modules. Since  $R $ is right coherent, the exact complex $I^{\bullet +}$  consists of   flat $R$-modules. We will prove that if $N$ is an $R^{\op}$-module with $\fd (N) =m < \infty$, then $N\otimes _R I^{\bullet +}$ is acyclic. The case $m=0$ is clear. Let  $m>0$ be, then there is an exact sequence $0 \to F_1 \to F \to N \to 0$, with $F \in \Flat (R^{\op})$ and $\fd (F_1) \leq m-1$, thus we obtain the exact sequence 
$$0 \to F_1 \otimes I^{\bullet +}  \to F \otimes I^{\bullet +} \to N \otimes I^{\bullet +} \to 0$$
we know that $F \otimes I^{\bullet +}$ is acyclic and $F_1 \otimes I^{\bullet +}$ is acyclic by induction hypothesis. It follows that $N \otimes I^{\bullet +}$ is acyclic. Now let  $A \in \A$ with $ \fd (A) < \infty$. It is now clear that $A \otimes I^{\bullet +}$ is acyclic. 
\end{proof}

\begin{cor}
Consider $R$ a right coherent ring and suppose that $\fd (A) < \infty$ for all $A \in \A \subseteq \Modu (R^{\op})$, then for all $G \in \GI_{(\A, \Inj (R^{\op}))}$ it is satisfied that $G^{+} \in \GF _{(\Flat , \A)}$.
\end{cor}

The following result will be very useful, as it allows us to use the Auslander-Buchweitz approximation theory.
\begin{pro} \label{cogenera}
Let $ \A \subseteq \Modu (R^{\op})$ be a class such that $\Inj (R^{\op}) \subseteq \A$, with $R$ a right coherent ring and suppose that $\GF_{(\Flat, \A)}$ is closed under extensions. Then the intersection $\Flat (R) \cap \Flat (R)^{\ortogonal }$ is a  $\GF_{(\Flat, \A)} $-injective relative cogenerator for $\GF_{(\Flat, \A)}$.

\end{pro}

\begin{proof}
Consider $M \in \GF _{(\Flat, \A)}$. We know that there is an exact complex by flat $R$-modules 
$$\cdots \to L_1 \to L_0 \to L^0 \to L^1 \to \cdots $$
which is  $(-\otimes _R \A)$-acyclic and with $\Ker(L^0 \to L^1) = M$.  Lets take the exact sequence  $0 \to M \to L^0 \to N \to 0$. Since $(\Flat (R), \Flat (R)^{\ortogonal })$ is a complete (and hereditary) cotorsion pair and $\Flat (R)$ is closed under extensions, for $ L ^{0}$ there is an exact sequence $0 \to L^0 \to K \to L \to 0$ with $K \in \Flat (R) \cap \Flat (R) ^{\ortogonal}$ and $L \in \Flat  (R) $. Consider the following p.o. diagram

 $$\xymatrix{
   M\; \ar@{=}[d] \ar@{^{(}->}[r]  & L_{0_{}} \ar@{^{(}->}[d]  \ar@{>>}[r] & N _{} \ar@{^{(}->}[d]    \\
  M\;  \ar@{^{(}->}[r]  & K  \ar@{>>}[d] \ar@{>>}[r] & M'  \ar@{>>}[d] \\
       & L \ar@{=}[r]&  L   }$$
since $N\in \GF _{(\Flat, \A)}$,  $L \in\Flat  (R)$ and $\Flat (R)$ is closed under extensions, by using \cite[Lemma 2.11]{Wang19} we have that $M ' \in \GF _{(\Flat, \A)}$, thus the exact sequence $0 \to M \to K \to M' \to 0 $ is such that $K \in \Flat (R) \cap \Flat (R) ^{\ortogonal}$ and $M ' \in \GF _{(\Flat, \A)}$. 
We still have to prove that $\GF _{(\Flat, \A)} \ortogonal (\Flat (R) \cap \Flat (R) ^{\ortogonal})$, to do this let's choose $X \in \Flat (R) \cap \Flat (R) ^{\ortogonal}$, we always have the pure exact sequence $\xi : 0 \to X \to X ^{++} \to X^{++ }/ X \to 0$, where $X \in \Flat(R)$ and  $X ^{++} \in \Flat (R)$ ($R$ is right coherent). Since $\Flat (R)$ is closed by pure quotients we obtain that $X ^{++} / X \in \Flat (R)$, so that the exact sequence $\xi$ splits, and thus for $T \in \GF _{(\Flat, \A)}$ we have  
\begin{eqnarray*}
\Ext ^{i} _R(T, X) \oplus \Ext ^{i}_R (T, X^{++} / X) & \cong & \Ext _R ^{i} (T,X^{++})\\
& = &  \Ext_R ^{i} (T, \Hom _{\mathbb{Z} }(X^{+} , \mathbb{Q} / \mathbb{Z})) \\
 & \cong & \Hom _{\mathbb{Z} } (\mathrm{Tor} ^{R} _i (X^{+}, T),\mathbb{Q} / \mathbb{Z} ) =0
\end{eqnarray*}
the last term is zero, since $\A \top \GF _{(\Flat, \A)}$ by \cite[Lemma 2.3]{Estrada18}.
\end{proof}

  From the hypothesis of  the previous results we see that it is important to know when the class $\GF _{(\Flat, \A)}$ is closed under extensions, this has been studied recently.  From \cite[Proposition 7]{Alina20} we know that this occurs when the class $\A$ is \textit{semi-definable} and $\Inj (R^{\op}) \subseteq \A$. Also from \cite[Corollary 5.3]{Gill21} when $(\Le, \A)$ is a semi-complete duality pair then $(\GF _{(\Flat, \A)} (R), \GF _{(\Flat, \A)} (R) ^{\ortogonal _1})$ is a perfect cotorsion pair and thus, $\GF _{(\Flat, \A)} (R)$ is closed under extensions.  
  
  The  \textbf{$(\Flat(R),\A)$-Gorenstein  flat dimension of $M\in \Modu (R)$} is defined by 
$$\Gfd_{(\Flat ,\A)}(M):=\resdim_{\GF_{(\Flat ,\A)}}(M),$$
and for  any class $\mathcal{Z}\subseteq\A,$   $\Gfd_{(\Flat,\A)}(\mathcal{Z}):=\sup \{\Gfd_{(\Flat,\A)}(Z)\;:\;Z\in\mathcal{Z}\}.$   

\begin{teo}\label{AB4}
Let $  \A \subseteq \Modu (R^{\op})$ be,  and suppose that $\GF _{(\Flat, \A)}$ is closed under extensions.
Then  for all $C\in \Modu (R)$ with $\Gfd_{(\Flat, \A)} (C)=n<\infty$,   the following statements are true.
\begin{itemize}
\item[(a)] There are exact sequences in  $\Modu (R)$
\begin{center}
$0\to K\to X\stackrel{\varphi}\to C\to 0,$
\end{center}
with $\resdim_{\Flat} (K)= n-1$ and $X\in \GF _{(\Flat, \A)},$ and
\begin{center}
$0\to C\stackrel{\varphi'}\to H\to X'\to 0,$
\end{center}
with $\resdim_{\Flat}(H)= n$ and $X'\in \GF _{(\Flat, \A)}.$
\item[(b)] If $R$ is a  right coherent ring  and $\Inj (R^{\op}) \subseteq \A$, then
  \begin{itemize}
  \item[(i)] $\varphi:X\to C$ is an $\GF _{(\Flat, \A)}$-precover and $K\in \GF _{(\Flat, \A)}^{\perp},$
  \item[(ii)] $\varphi':C\to H$ is an $(\Flat \cap \Flat ^{\ortogonal })^{\wedge}$-preenvelope and $X'\in{}^\perp((\Flat \cap \Flat ^{\ortogonal })^{\wedge}).$
  \end{itemize}
  \end{itemize}
\end{teo}

\begin{proof}
The item  (a) it follows from \cite[Thm 2.8 (a)]{BMS} while the equality $\mathrm{resdim} (H) = n$  it holds since $\Flat (R)$ is a resolving class.  

The item (b) it follows from Proposition \ref{cogenera}  and  \cite[Thm 2.8 (b)]{BMS}.
\end{proof}

The following result makes use of the fact that the flat dimension  $\fd (-)$ coincides with  $\resdim _{\Flat} (-)$ \cite[Proposition 8.17]{Rotman}.
\begin{pro}
Let $\A \subseteq \Modu (R^{\op}) $ be,  and suppose that  $\GF _{(\Flat, \A)}$ is closed under extensions. Then for all  $M \in \A ^{+} $ the equality $\mathrm{fd} (M) = \Gfd_{(\Flat, \A)} (M)$ is given.
\end{pro}

\begin{proof}
Let us suppose that   $\Gfd _{(\Flat, \A)}  (M) < \infty$,  by Theorem \ref{AB4} (a), there is an exact sequence $\eta : 0 \to M \to X \to G \to 0$, such that $\mathrm{fd} (X) = \Gfd_{(\Flat, \A)} (M)$ and $G \in \GF _{(\Flat, \A)}$. Since $M \in \A ^{+}$ then $M = A ^{+}$ for some $A \in \A$, and thus 
$$\Ext ^{i} _R (G, M) = \Ext ^{i} _R (G, A^{+}) \cong \Hom _{\mathbb{Z}} (\Tor _i ^{R} (A, G) , \mathbb{Q} / \mathbb{Z}).$$
Now  $\A \top \GF _{(\Flat, \A)}$  by \cite[Lemma 2.10]{Wang19}, thus the sequence  $\eta$ splits, therefore $$\mathrm{fd} (M) \leq \mathrm {fd} (X)  = \Gfd_{(\Flat, \A)} (M).$$

Finally the inequality $\mathrm{fd} (M) \geq \Gfd_{(\Flat, \A)} (M)$ is always true.
\end{proof}

\begin{pro} \label{PropositionF-I}
Consider a  class $ \A \subseteq \Modu (R^{\op})$. Then for all  $M \in \GF _{(\Flat, \A)}$, it holds that $M ^{+} \in \GI_{ (\A, \Inj)}$.
 \end{pro}
 \begin{proof}
 Take $F ^{\bullet} \in \Ch (\Flat (R)) $ an exact complex of flat $R$-modules such that the complex  $A \otimes _R F^{\bullet}$ is acyclic for all  $A \in \A$.  Since $(-) ^{+}$ is an exact functor and  by \cite[Theorem 2.1.10]{EnJen00} the following isomorphism is natural
 $$(A \otimes_R F^{\bullet}) ^{+}  \cong \Hom _R (A, (F^{\bullet} ) ^{+}), $$
 we have that the last term is acyclic if and only if $A \otimes_R F^{\bullet} $ is acyclic.  By Lambek's Theorem \cite[Proposition 3.54]{Rotman} the complex $(F^{\bullet}) ^{+}$ is of injective $R^{\op}$-modules and also is acyclic, thus is $\Hom _R (\A, - )$-acyclic.
 \end{proof}

\begin{pro} \label{DimensionCoincide}
Let $ \A\subseteq \Modu (R^{\op}) $ be a class such that the pair $(\A , \Inj (R^{\op}))$ is GI-admissible in $\Modu (R^{\op})$, then for all $M \in \Modu (R)$ with $\pd (M^{+}) < \infty$ it is satisfied that $\mathrm{fd} (M)  = \Gfd_{(\Flat, \A)} (M)$.  
\end{pro}

\begin{proof}
Since   $(\A, \Inj(R^{\op}))$ is  GI-admissible, by using the dual of \cite[Lemma 3.3]{Becerril21} we have in $\Modu (R^{\op})$ that
$$\id (M^{+}) = \Gid _{(\A, \Inj)} (M^{+}),$$ 
and is true that $\fd (M) = \id (M ^{+})$ (for every ring $R$) and it is also true   (we prove this below) that $\Gid _{(\A, \Inj)} (M^{+}) \leq \Gfd _{(\Flat, \A)} (M)$, thus we obtain the inequality
$$ \fd (M) \leq \Gfd_{(\Flat, \A)} (M) $$
but the opposite inequality always occurs, since $\Flat (R) \subseteq \GF _{(\Flat, \A)}$.

It remains to prove that $\Gid _{(\A, \Inj)} (M^{+}) \leq \Gfd _{(\Flat, \A)} (M)$ but it follows from  Proposition \ref{PropositionF-I} .
\end{proof}
 
A natural question is whether all $(\Le, \A)$-Gorenstein projective $R$-modules are $(\Le, \A)$-Gorenstein flat, for now we will prove this  for a symmetric duality pair $\{ \Le,  \A \}$. Thus,  the inequality $\Gfd _{(\Flat, \A)} (M) \leq \Gpd _{(\Proj , \Le)} (M)$ will be true for all $M \in \Modu (R)$.

\begin{pro} \label{ProyectivosPlanos}
Let  $( \Le , \A ) \subseteq \Modu (R) \times \Modu (R^{\op})$ be such that $\A ^{+} \subseteq \Le$. Then  the containment $\GP _{(\Proj , \Le)} \subseteq \GF _{(\Flat, \A)}$ is given. 
\end{pro}

\begin{proof}
 Consider an exact complex $\mathbf{P}$ of projective $R$-modules  which is  $\Hom _R (-, \Le)$-acyclic, we will see below that for all $A \in \A$ it is satisfied that $A \otimes _R \mathbf{P}$ is acyclic.  By \cite[Theorem 2.1.10]{EnJen00} we have the following natural isomorphism 
 $$\Hom _{\mathbb{Z}} (A\otimes _R \mathbf{P}, \mathbb{Q} / \mathbb{Z} ) \cong \Hom _R (\mathbf{P}, A ^{+}).$$
 Since for all $A\in \A$ is satisfied that $ A ^+ \in \Le$, we have that  $\Hom _{R} (\mathbf{P}, A^+)$ is acyclic for all $A \in \A$. Therefore $A \otimes _R \mathbf{P}$ is acyclic for all $A \in \A$.
\end{proof}

\begin{lem}\label{DimensionesGF}
Let $ \A\subseteq \Modu (R^{\op}) $ be, and suppose that the class $\GF_{(\Flat, \A)}$  is closed under extensions. For all $M \in \GF ^{\gorro} _{(\Flat, \A)} $ the following statements are equivalent.
\begin{itemize}
\item[(i)] $\Gfd _{(\Flat, \A)} (M) \leq n$,
\item[(ii)] If $0 \to K_n\to H_{n-1}  \to \cdots \to H_{0} \to M \to 0$ is an exact sequence with  $H_{i} \in \GF _{(\Flat, \A)}$, then $K^n \in \GF _{(\Flat, \A)}$.
\end{itemize}
\end{lem}
\begin{proof}
Since the coproduct of flat $R$-modules is flat and we work in  $\Modu (R)$, then the coproduct of exact complexes with flat components is an exact complex of flat components, then as the tensor commutes with the coproducts, we have that the class $\GF_{(\Flat, \A)}$ is closed by coproducts. It follows from the hypotheses that  $\GF_{(\Flat, \A)}$ is a resolving class, then by the Eilenberg's Swindle we obtain that the class $\GF_{(\Flat, \A)}$ is closed under direct summands. Finally the result follows from the Auslander-Bridger's Lemma \cite[Lemma 3.12]{AuBri69}.
\end{proof}
We have tha the following Lemma, which is similar to  \cite[Lemma 2.5]{Bennis09}.

\begin{lem} \label{Bennis2.9}
Let  $ \A\subseteq \Modu (R^{\op}) $ be, such that $\Inj (R^{\op}) \subseteq \A$. The following statements are equivalent.
\begin{itemize}
\item[(1)] $\GF _{(\Flat, \A)}$ is closed under extensions,
\item[(2)] The class $\GF _{(\Flat, \A)}$ is pre-resolving,
\item[(3)] For all exact sequence of  $R$-modules $0 \to G_1 \to G_0 \to M \to 0$ with $G_1 , G_0 \in \GF _{(\Flat, \A)}$  if $\Tor ^R _1 (\A, M ) =0$, then $M \in \GF _{(\Flat, \A)}$.
\end{itemize}
\end{lem}
\begin{proof}
We will prove $(1) \Rightarrow (3)$. The remaining implications are similar to \cite[Lemma 2.5]{Bennis09}. Consider the exact sequence  $0 \to G_1 \to G_0 \to M \to 0$ with $G_1 , G_0 \in \GF _{(\Flat, \A)}$ and such that $\Tor ^R _1 (\A, M ) =0$. Since $ G_1 \in \GF _{(\Flat, \A)}$, then there is an exact sequence $0 \to G_1 \to F_1 \to H \to 0$, with $F_1 \in \Flat (R)$ and $ H \in \GF _{(\Flat, \A)}$. We have the following p.o. diagram
$$\xymatrix{ G_{1_{}}\ar@{^{(}->}[r] \ar@{^{(}->}[d] & G_{0_{}}   \ar@{^{(}->}[d] \ar@{>>}[r] & M  \ar@{=}[d] \\
  F_1  \ar@{>>}[d] \ar@{^{(}->}[r] & D  \ar@{>>}[d] \ar@{>>}[r]& M  \\
 H  \ar@{=}[r]   & H   &     }$$
from the sequence $0 \to G_0 \to D \to H \to 0$ we have that $G_0 , H \in \GF _{(\Flat, \A)}$ implies $D \in \GF _{(\Flat, \A)}$ (by hypothesis), thus there is an exact sequence $0 \to D \to F \to G \to 0$, with $F \in \Flat$ and $G \in \GF _{(\Flat, \A)}$. Consider the following p.o. diagram
$$\xymatrix{  F_1  \ar@{=}[d] \ar@{^{(}->}[r]  & D _{} \ar@{^{(}->}[d] \ar@{>>}[r] & M_{} \ar@{^{(}->}[d]    \\
  F_1  \ar@{^{(}->}[r]  & F  \ar@{>>}[d] \ar@{>>}[r] & F' \ar@{>>}[d] \\
       & G  \ar@{=}[r]&  G   }$$
We will show that $F' \in \Flat$. Consider the exact sequence $0 \to M \to F' \to G \to 0$. For $A \in \A$ we have the exact sequence
$$ 0 = \Tor ^R _1 (A,M) \to \Tor ^R _1 (A, F') \to \Tor^R  _1 (A,G) =0,$$
thus $\Tor^R  _1 (A, F') =0$. For another hand, we also have the exact sequence  $ 0 \to F_1 \to F \to F'\to 0$, and from this we obtain the exact sequence  $\beta : 0\to (F') ^{+} \to F^{+} \to (F_1) ^{+} \to 0 $ with $ (F_1) ^{+}, F^{+} \in \Inj (R^{\op})$. Now we have the following isomorphism 
$$\Ext _R ^1 ((F_1) ^{+}, (F' ) ^{+}) \cong (\Tor ^R _1 ((F_1) ^{+}, F') )^{+} =0, \mbox{ since } F_1 ^{+} \in  \Inj (R^{\op}) \subseteq \A$$ 
therefore the exact sequence $\beta$  splits, thus $(F') ^{+}$ is a direct summand of an injective $R^{\op}$-module, and from this   $(F') ^{+}$ is an injective $R^{\op}$-module, then by Lambek's Theorem $F'$ is a flat $R$-module. Finally  we have the exact sequence$0 \to M \to F' \to G \to 0$ with $G \in \GF _{(\Flat, \A)}$ and $F' \in \Flat$, and thus $M \in \GF _{(\Flat, \A)}$.
\end{proof}

The following result is a generalization of \cite[Theorem 2.8]{Bennis09}.

\begin{teo} \label{Teo2.8Bennis09}
Let  $ \A\subseteq \Modu (R^{\op}) $ be, and suppose that $\GF_{(\Flat, \A)}$ is closed under extensions. Consider the following statements for $M \in \Modu (R)$.
\begin{itemize}
\item[(1)] $\Gfd_{(\Flat, \A)} (M) \leq n$.
\item[(2)] $\Gfd _{(\Flat, \A)}(M) < \infty$ and $\Tor _{i} ^{R} (A, M) =0$ for all $i > n$ and all $ A \in \A \subseteq \Modu (R^{\op})$.
\item[(3)] $\Gfd_{(\Flat, \A)} (M) < \infty$ and $\Tor   _{i} ^{R} (E, M) =0$ for all $i > n$ and all $E  \in \A ^{\cogorro} \subseteq \Modu (R^{\op})$. 
\end{itemize}
Then $(1) \Rightarrow (2) \Rightarrow (3)$. Furthermore, if $\Inj (R^{\op}) \subseteq \A$ then all statements are equivalent.
\end{teo}

\begin{proof}
$(1) \Rightarrow (2)$. By induction over $n$. The case $n =0 $ is satisfied by \cite[Lemma 2.5]{Estrada18}. Thus, we may assume that $n >0$, There is therefore an exact sequence $0 \to K \to G \to M \to 0$, with $G \in \GF _{(\Flat, \A)}$ and $\Gfd _{(\Flat, \A)} (K) = n-1$. We know that for all $A \in \A$ it is satisfied that $\Tor ^R _{i} (A, G) =0$ for all $i > 0$ and  that $\Tor ^R _{i} (A, K) =0$ for all $i > n-1$ (by induction). Then we use the long exact sequence $\Tor ^R _{i +1}( A, G) \to \Tor ^R _{i +1} (A,M) \to \Tor ^R _{i } (A,K)$ to conclude that $\Tor ^R _{i +1} (A, M) =0$ for all $ i > n-1$.\\

$(2) \Rightarrow (3)$ It follows by a shifting argument.

$(3) \Rightarrow (1)$ Let us suppose that $\Inj (R^{\op}) \subseteq \A$. Since $\Gfd _{(\Flat, \A)} (M)$ is finite, from Lemma \ref{DimensionesGF} we can choose, for some $m> n$ an exact sequence
$$0 \to G_m \to \cdots \to G_0 \to M \to 0$$
with $G_0 , \dots , G_{m-1} \in  \Flat (R)$  and $G_m \in \GF_{(\Flat, \A)} $. Lets take $K_n := \Ker (G_{n-1} \to G_{n-2})$, the goal is to prove that $K_n \in \GF _{(\Flat, \A)}$. 

We split the sequence $0 \to G_m \to \cdots \to G_n \to K_n \to 0$, in short exact sequences $0 \to H_{i+1} \to G_i \to H_i \to 0$ for $i \in \{n , \dots, m-1 \}$, where $H_n = K_n $ y $H_m = G_m$. Thus, let us consider the exact sequence $0 \to H_m \to G_{m-1} \to H_{m-1} \to 0$. We claim that  $H_ {m-1}  \in \GF _{(\Flat, \A)}$. From the exact sequence $0 \to H_{m-1} \to \cdots \to G_0 \to M \to 0 $, we have that for all $E \in \A \subseteq  \Modu (R ^{\op})$ and all $i >0 $ we have the isomorphism $\Tor^R _i (E, H_{m-1}) \cong \Tor^R _{(m-1)+i} (E,M)=0$. Therefore, for $H_{m-1}$ there is an exact sequence $0 \to G_m \to G_{m-1} \to H_{m-1} \to 0$ and $\Tor ^R_i (E, H_{m-1}) =0$ for all $i >0$. Thus by Lemma \ref{Bennis2.9} we have that $H_{m-1} \in \GF _{(\Flat, \A)}$. Such  argument can be repeated to show that $H_{m-2} , \dots , H_n =K_n$ are all in $ \GF _{(\Flat, \A)}$. 
\end{proof}

\section{Finitistic Gorenstein flat \\and \\weak Gorenstein global dimensions}\label{S4}

Now that we have enough tools, we are ready to study other kinds of dimensions.
Given a class $\A \subseteq \Modu (R^{\op}) $, we define the \textbf{left weak Gorenstein global dimension} relative to the pair $(\Flat (R), \A)$ as follows:
$$\mathrm{l.w.Ggl} _{(\Flat, \A)}(R):= \sup \{\Gfd _{(\Flat, \A)} (M)| M \in \Modu (R)  \} .$$
From  Theorem \ref{Teo2.8Bennis09} the following result is easily deduced.

\begin{pro} \label{EquivalenciaGlobal}
Let  $ \A\subseteq \Modu (R^{\op}) $ be, with $\GF_{(\Flat, \A)}$ closed under extensions and $\mathrm{l.w.Ggl} _{(\Flat, \A)}(R)  <\infty$. Consider the following assertions.
\begin{itemize}
\item[(1)] $\mathrm{l.w.Ggl} _{(\Flat, \A)}(R) \leq n <\infty$,
\item[(2)]  $\fd (A) \leq n$ for all $A \in \A \subseteq \Modu (R^{\op})$,
\item[(3)] $\fd (E) \leq n$ for all $E \in \A ^{\cogorro} \subseteq \Modu (R^{\op})$.
\end{itemize}
Then $(1) \Rightarrow (2) \Rightarrow (3)$. Furthermore, if $\Inj (R^{\op}) \subseteq \A$ then all assertions are equivalent i.e. the following equalities are true
$$\mathrm{l.w.Ggl} _{(\Flat, \A)}(R) = \fd (\A) = \fd (\A^{\cogorro}).$$

\end{pro}

To further study of the weak global dimension, we need more tool. To do so, let us consider a complementary dimension, which is an adaptation of the \textit{copure flat dimension} defined by E. Enochs and O. M. G. Jenda \cite{EnJen93}.

\begin{defi}
Consider $\A \subseteq  \Modu (R^{\op}) $ and $M \in \Modu (R)$. \textbf{The $\A$-flat co-dimension of $M$}, denoted $c\fd _{\A} (M)$ is the largest positive integer $n$ such that $\Tor _{n} (A, M) \not= 0$, for some $A \in \A$, this is
$$c\fd _{\A} (M) := \sup \{n: \Tor  ^{R} _n (A, M) \not = 0 \mbox{ for some } A \in \A \}.$$
\end{defi}

Note that if $\GF_{(\Flat, \A)}$ is closed under extensions, then by Theorem \ref{Teo2.8Bennis09} we have that for all  $M \in \Modu (R)$ the following inequality holds
$$c\fd _{\A} (M) \leq \Gfd _{(\Flat, \A)} (M).$$

At this point we see necessary to consider another kind of duality pair $(\Le, \A)$ with the property  that $(^{\ortogonal _1} \A , \A)$ will be a cotorsion pair.

\begin{defi}\cite[Definition 3.2]{Wang20} \label{Bi-complete}
A duality pair $(\Le, \A)$ in $\Modu (R)$ is called \textbf{bi-complete} if $(\Le, \A)$ is a complete duality pair such that the pair $(^{\ortogonal_1} \A, \A)$ forms a hereditary cotorsion pair cogenerated by a set.
\end{defi}

The theory developed in \cite[\S 3.1]{Wang20} shows that the pairs $(\Le, \Le ^{\ortogonal _1})$ and $(^{\ortogonal _1 }\A, \A)$ are complete and hereditary cotorsion pairs. Thus the coming results with $(^{\ortogonal _1 }\A, \A)$ as hypothesis can be thought over a bi-complete duality pair. We will give examples of the above mentioned pairs in Example \ref{E-Bi-complete} below.

\begin{lem} \label{FinitudPlana}
Let $\A \subseteq \Modu (R ^{\op})$ be, such that $(^{\ortogonal _1} \A, \A)$ is a complete cotorsion pair,  and suppose that $\GF _{(\Flat, \A)}$ is closed under extensions. Then for all   $M \in \Flat (R) ^{\gorro}$, it holds   $$c\fd _{\A} (M)  = \Gfd _{(\Flat, \A)} (M) = \fd (M).$$
\end{lem}
\begin{proof}
By Theorem \ref{Teo2.8Bennis09} we already know that $c\fd _{\A} (M) \leq \Gfd _{(\Flat, \A)} (M) \leq \fd (M)$, since $\Flat (R) \subseteq \GF _{(\Flat, \A)}$. In particular we have that $c\fd _{\A} (M) \leq \fd (M)$. Suppose now that  $\fd (M) := n < \infty$. Then, there is  $N \in \Modu (R^{\op}) $ such that $\Tor ^R _n (N,M) \not= 0 $, and for such $N$ there is an exact sequence $\gamma :0 \to N \to A \to C \to 0$, with $A \in \A$ and $C \in {^{\ortogonal _1} \A}$.  Applying $- \otimes _R M$ to the sequence $\gamma $ we obtain the exact sequence 
$$0=\Tor^R _{n+1} (C, M) \to \Tor^R _n (N, M) \to \Tor^R _n (A, M)$$
where the left-hand side is zero since $\fd (M) =n$, thus $\Tor_{n} (A, M) \not= 0 $.  Which implies that $c\fd _{\A} (M) \geq n = \fd (M)$. This is $c\fd _{\A} (M) = \fd (M) $
\end{proof}

We are interested in study the (left and right) weak global dimension relative to a   duality pair $(\Le, \A)$. Thus, for a class $\Le \subseteq \Modu (R)$ we define the  \textbf{right weak  Gorenstein global dimension} relative to the pair  $(\Flat (R^{\op}), \Le )$ as follows 
$$\mathrm{r.w.Ggl} _{(\Flat(R^{\op}), \Le )}(R):= \sup \{\Gfd _{(\Flat(R^{\op}), \Le )}  (M)| M \in \Modu (R^{\op})  \}.$$
From such definition we obtain a dual version of Proposition  \ref{EquivalenciaGlobal} as follows.

\begin{pro} \label{EquivalenciaGlobalR}
Let $\Le \subseteq \Modu (R) $ be, with $\GF _{(\Flat(R^{\op}), \Le )}$ closed under extensions in $\Modu (R^{\op})$ and $\mathrm{r.w.Ggl} _{(\Flat(R^{\op}), \Le )}(R) <\infty$. Consider the following assertions.
\begin{itemize}
\item[(1)] $\mathrm{r.w.Ggl} _{(\Flat(R^{\op}), \Le )}(R) \leq n <\infty$.
\item[(2)]  $\fd (L) \leq n$ for all $L \in \Le  \subseteq \Modu (R)$.
\item[(3)] $\fd (E) \leq n$ for all $E \in \Le  ^{\cogorro} \subseteq \Modu (R)$.
\end{itemize}
Then $(1) \Rightarrow (2) \Rightarrow (3)$. Furthermore, if $\Inj (R) \subseteq \Le$ then all assertions are equivalent i.e. the following equalities are true
$$ \mathrm{r.w.Ggl} _{(\Flat(R^{\op}), \Le )}(R) = \fd (\Le) = \fd (\Le ^{\cogorro}).$$

\end{pro}

\begin{pro} \label{UnderLimit}
Consider a pair of classes $(\Le,\A) \subseteq \Modu (R) \times \Modu (R^{\op})$ that  satisfies the following conditions
\begin{itemize}
\item[(1)] The pair $(^{\ortogonal _1} \A , \A)$ is a complete cotorsi\'on pair.
\item[(2)] The classes $\GF _{(\Flat (R^{\op}), \Le )}$  and  $\GF_{(\Flat, \A)}$ are closed under extensions.
\item[(3)]  The dimension $\mathrm{r.w.Ggl} _{(\Flat (R^{\op}), \Le)}(R) $ is finite.
\end{itemize}
 Then the inequality holds.
 $$\fd  (\Le) \leq \mathrm{l.w.Ggl} _{(\Flat(R), \A)}(R)$$
\end{pro} 

\begin{proof}
Indeed, since $\mathrm{r.w.Ggl} _{(\Flat (R^{\op}), \Le)}(R)$ is finite, by Proposition \ref{EquivalenciaGlobalR} we have that  $ \fd (\Le) \leq \mathrm{r.w.Ggl} _{(\Flat (R^{\op}), \Le)}(R)  < \infty$, then using Lemma \ref{FinitudPlana} we obtain for all $L \in \Le$  that $ \Gfd _{(\Flat (R) , \A)} (L) = \fd (L)$, this implies  $\fd (\Le) \leq \mathrm{l.w.Ggl} _{(\Flat(R), \A)}(R)$.
\end{proof}


In what follows, we recall for a ring $R$ the left finitistic flat  dimension $\mathrm{l.FFD} (R)$ and we define with respect to the pair  $(\Flat, \A)$ the \textbf{left finitistic Gorenstein flat dimension}.
\begin{center}
$l.\mathrm{FFD} (R) = \sup \{ \fd (M ) : M \in \Modu (R) \mbox{ such that } M \in \Flat (R) ^{\gorro} \}$,\\
$l.\mathrm{FGFD}_{(\Flat, \A)} (R) = \sup \{ \Gfd _{(\Flat, \A)} (M ) : M \in \Modu (R) \mbox{ such that } M \in \GF_{(\Flat, \A) } ^{\gorro}\}.$
\end{center}

We are interested in compare both finitistic dimensions, for this end we have the following result.

\begin{pro}\label{Desigualdad}
Let  $\A \subseteq \Modu (R^{\op})  $ be, and suppose that  $\GF _{(\Flat, \A)}$ is closed under extensions, then the following inequality holds
$$l.\mathrm{FGFD}_{(\Flat, \A)} (R) \leq  l.\mathrm{FFD} (R)  .$$
\end{pro}
\begin{proof}
We can suppose that  $n:= l.\mathrm{FFD} (R) < \infty$. Consider $ M \in \GF ^{\gorro} _{(\Flat, \A)}$, by Theorem  \ref{AB4} we know that there is an exact sequence 
$ 0 \to M \to H \to X' \to 0$ with $X' \in \GF _{(\Flat,  \A)}$ and $\fd (H) = \Gfd _{(\Flat, \A)} (M)$, but by hypothesis $\fd (H) \leq n$. This is $\Gfd _{(\Flat, \A)} (M) \leq n$ for all $M \in  \GF ^{\gorro} _{(\Flat, \A)}$.  
\end{proof}

\begin{cor}
Let $\A \subseteq \Modu (R ^{\op})$ be, such that $(^{\ortogonal _1} \A, \A)$ is a complete cotorsion pair,  and suppose that $\GF _{(\Flat, \A)}$ is closed under extensions. Then the following equality holds $$l.\mathrm{FFD} (R)  = l.\mathrm{FGFD}_{(\Flat, \A)} (R).$$
\end{cor}

\begin{proof}
For the inequality   $l.\mathrm{FFD} (R) \leq l.\mathrm{FGFD} _{(\Flat, \A)} (R) $ is enough to prove that $\Gfd _{(\Flat, \A)} (M) = \fd (M)$ for all $M \in \Flat (R) ^{\gorro}$, which is true by Lemma \ref{FinitudPlana}. 

The other inequality it follows from Proposition \ref{Desigualdad}.
\end{proof}

\begin{rk} \label{GlobalFinitista}
Observe that if $\mathrm{l.w.Ggl} _{(\Flat(R), \A)}(R) < \infty$ then 
$$\mathrm{l.w.Ggl} _{(\Flat, \A)}(R)  =  l.\mathrm{FGFD}_{(\Flat, \A)} (R)$$
 and if $\Inj (R^{\op}) \subseteq \A$ and $\GF _{(\Flat , \A)}$ is closed under extensions then by Proposition \ref{EquivalenciaGlobal}  $\fd (\A) = \mathrm{l.w.Ggl} _{(\Flat, \A)}(R)  = l.\mathrm{FGFD}_{(\Flat, \A)} (R) $ and if in addition $(^{\ortogonal _1} \A , \A)$ is a complete cotorsion pair, then by the previous Corollary we have the equalities $$\fd (\A) = \mathrm{l.w.Ggl} _{(\Flat, \A)}(R)  = l.\mathrm{FGFD}_{(\Flat, \A)} (R)  = l.\mathrm{FFD} (R) .$$
\end{rk}

Remember that a semi-complete (or perfect) duality pair $(\Le, \A)$ satisfies that $ \Inj (R^{\op}) \subseteq  \A$ \cite[Remark 2.6]{Gill21} and that the pair $(\A, \Inj (R^{\op}))$ is GI-admissible. 
Note also that the containment in the hypothesis of the following result is satisfied when $\Inj (R^{\op}) ^{\cogorro} \subseteq \Proj (R^{\op}) ^{\gorro}$.
 
\begin{pro} \label{FinitistaFrobenius}
Let $ \A\subseteq \Modu (R^{\op}) $ be a class such that the pair $(\A , \Inj (R^{\op}))$ is GI-admissible in $\Modu (R^{\op})$. If $(\Flat (R) ^{\gorro}) ^{+ } \subseteq \Proj (R^{\op}) ^{\gorro}$  then we have the inequality $$l. \mathrm{FFD} (R) \leq   l.\mathrm{FGFD}_{(\Flat, \A)} (R).$$
\end{pro}
\begin{proof}
 We can assume that $ l.\mathrm{FGFD}_{(\Flat, \A)} (R) := n < \infty$. Lets  take $T \in \Flat (R) ^{\gorro}$, then by hypothesis $T ^{+} \in \Proj (R^{\op}) ^{\gorro}$.  Thus, by Proposition \ref{DimensionCoincide} we have that $ \Gfd _{(\Flat, \A)}(T) = \fd (T) < \infty$, but $\Gfd _{(\Flat, \A)}(T) \leq n$. This is $\fd (T) \leq n$ for all $T \in \Flat (R) ^{\gorro}$ .
\end{proof}

\begin{cor}
Let $ \A\subseteq \Modu (R^{\op}) $ be a class such that the pair $(\A , \Inj (R^{\op}))$ is GI-admissible in $\Modu (R^{\op})$ with $(\Flat (R) ^{\gorro}) ^{+ } \subseteq \Proj (R^{\op}) ^{\gorro}$ and suppose that  $\GF _{(\Flat, \A)}$ is closed under extensions. Then the following equality holds 
$$l. \mathrm{FFD} (R) =   l.\mathrm{FGFD}_{(\Flat, \A)} (R).$$
\end{cor}

\begin{proof}
It follows from Propositions \ref{Desigualdad} and  \ref{FinitistaFrobenius}.
\end{proof}
Note that  similar remark to Remark \ref{GlobalFinitista} can be made from the previous corollary but with slightly different conditions.

 \section{Balance situations.} \label{S5}
Sometimes is possible to consider a duality pair $(\Le, \A)$ in $\Modu (R)$ such that $(_{R^{\op}} \Le, \A _{R^{\op}})$ is also a duality pair, but in $\Modu (R^{\op})$. This is the case in \cite[\S 4]{Wang20}, when it is established that $(\Le, \A)$ will be a bi-complete duality pair over $\Modu (R)$ if $(\Le, \A)$ is a bi-complete duality pair in $\Modu (R)$ and also $(_{R^{\op}} \Le, \A _{R^{\op}})$ is a bi-complete duality pair in $\Modu (R^{\op})$ (see \cite[Example 4.3]{Wang20} for examples). In what follows this will be the meaning of a bi-complete duality pair, thus we will use freely the results developed in \cite[\S 4]{Wang20}.  As discussed above, for a duality pair $(\Le, \A)$ in $\Modu (R)$, sometimes is possible to consider the class $\Le$ as one in $\Modu (R^{\op})$, and $\A$ as one in $\Modu (R)$. We use  this property within the hypotheses of the following results mentioning it explicitly in each case.\\
Let us consider the classes of $R$-modules $\GF _{(\Flat, \A)}$ and  $R^{\op}$-modules $\GI _{(\A, \Inj (R^{\op}))}$ relative to a pair of duality $(\Le, \A)$. We know that if $(\Le, \A)$ is a perfect or semi-complete duality pair, then  $(\Proj (R), \Le)$ is GP-admissible $\Modu (R)$, thus by \cite[Theorem 4.2 (a)]{BMS} there exist a left proper $\GP _{(\Proj , \Le)}$-resolution for each $M \in \GP _{(\Proj , \Le)} ^{\gorro}$. 
 Let's see now that if $\A$ can be considered in $\Modu (R)$ and $\Le$ in $\Modu (R^{\op})$, then $- \otimes_R - $ is balanced over $\GP _{(\Proj (R^{\op}), \Le)} ^{\gorro} \times \GP _{(\Proj , \Le)} ^{\gorro}  $ by $\GP _{(\Proj (R^{\op}) , \Le)} \times \GP _{(\Proj , \Le)}  $. We will need the following Lemma.
 
 \begin{lem}\label{LemaBalance} Lets take $\Le, \A \subseteq \Modu (R)$ a pair of classes such that $(\Proj (R), \Le)$ is a GP-admissible pair. If  $M \in \GP _{(\Proj , \Le)} ^{\gorro}$,  then any left proper $ \GP _{(\Proj , \Le)}$-resolution of $M$ is $\Hom _{R} (-,\GI _{(\A, \Inj)})$-acyclic.
   \end{lem}
   
   \begin{proof}
  Let $\mathbf{X}  (M) \to M $  be a left proper $\GP _{(\Proj , \Le)}$-resolution of $M$.  Since it breaks down in short exact sequences, is enough to show that $\Hom _{R} (\eta,Y)$ is exact for all $Y \in \GI _{(\A, \Inj (R))}$ and for all short exact sequence $\eta : 0 \to K \stackrel{i} \to X \to M \to 0$, where $X \to M $ is a special $\GP _{(\Proj , \Le)}$-precover for $M \in \GP _{(\Proj , \Le)}^{\gorro}$ with $K \in \Proj(R) ^{\gorro}$.    Therefore, we will prove that $\Hom _{R} (i,Y): \Hom _{R} (X,Y) \to \Hom _{R} (K,Y)$ is an epimorphism for all $Y \in \GI _{(\A, \Inj)}$.
   By definition of the class $\GI _{(\A, \Inj)}$ we know that for $Y \in \GI _{(\A, \Inj)}$ there is an exact sequence $\gamma :0 \to Y' \to V\stackrel{\rho}  \to Y  \to 0$, with $Y' \in \GI _{(\A, \Inj)}$ and $V \in \Inj (R)$. Consider the following commutative diagram:
   $$\xymatrix{
 \Hom _{\A} (X,V) \ar[r]^{(X,\rho)} \ar[d]^{(i,V)}  &  \Hom _{\A} (X,Y)  \ar[d]^{(i,Y)} \\
 \Hom _{\A} (K,V)  \ar[r]^{(K,\rho)} &  \Hom _{\A} (K,Y)  }$$
We know that $\Ext ^1 _{\A} (\GP _{(\Proj , \Le)}^{\gorro} , \Inj (R)) =0$, thus from the sequence $\eta$ we conclude that  $(i, V)$ is an epimorphism. It is enough to prove that  $\Ext ^1 _{R} (\Proj (R) ^{\gorro}, \GI _{(\A, \Inj)}) =0$, so from the sequence $\gamma$ we conclude that $(K, \rho)$ is an epimorphism. Therefore $(i , Y)$ is an epimorphism. Indeed take $N \in \Proj (R) ^{\gorro}$, then for $Y$ there is exact sequences $0 \to Y_{n+1} \to V_n \to Y_n  \to 0$ with $V_n \in \Inj (R)$ and $Y_n \in \GI _{(\A, \Inj)}$ for each $n \geq 0$ (with $Y_0 :=Y$). Applying $\Hom _R (N,-)$ have the exact sequence 
$$\Ext _{R} ^{i} (N,V_n) \to \Ext _{ R} ^{i} (N,Y_n) \to \Ext _{ R} ^{i+1} (N,Y_{n+1}) \to \Ext _{ R} ^{i+1} (N,V_n)$$ 
for all $i > 0$, with zero ends since $V_n \in \Inj (R)$, thus  $$\Ext _{ R} ^{i}(N,Y) \simeq \Ext _{ R} ^{i+n} (N, Y _{n}), \mbox{ for all } i>0 \mbox{ and for all } n\geq 0, $$
the term on the right-hand side is zero for  $i+n > \pd  (N)$, therefore  $\Ext _{ R} ^{i} (N,Y) =0$ for all $i>0$.
   \end{proof}

\begin{teo} \label{L-ABalan}
Let $R$ be a ring, $(\Le, \A)$ a perfect or semi-complete duality pair in $\Modu (R)$ such that $  \A $ can be considered in $\Modu (R)$, then the functor $\Hom _{R} (-,-)$ is right balanced over $\GP _{(\Proj , \Le)} ^{\gorro} \times \GI _{(\A, \Inj)} ^{\cogorro}$ by $\GP _{(\Proj , \Le)} \times \GI _{(\A, \Inj)}$.
\end{teo}

\begin{proof}
We know that  if $(\Le, \A)$ is perfect or semi-complete duality pair then the pairs $(\Proj , \Le)$,  $(\A, \Inj)$  are GP-admissible and GI-admissible respectively. The result follows from  Lemma \ref{LemaBalance} and its dual.
\end{proof}

 \begin{teo} \label{BalanF-P}
Lets take $\Le, \A \subseteq \Modu (R)$ a pair of classes such that $(\Proj (R), \Le)$ is a GP-admissible pair and  $ \mathbf{D} : \cdots \to G_1 \to G_0 \to M \to 0$ a left proper $\GP _{(\Proj , \Le)}$-resolution for the $R$-module $M \in \GP _{(\Proj , \Le)} ^{\gorro}$. Then the complex  $G \otimes _R \mathbf{D}$ is acyclic for all  $R^{\op}$-module $G \in \GF _{(\Flat (R^{\op}) , \A)} (R^{\op})$.
 \end{teo}
 
 \begin{proof}
We know that the complex  $G \otimes_R \mathbf{D} : \cdots \to G \otimes_R
 G_1 \to  G \otimes_R G_0 \to G \otimes_R X \to 0$ is acyclic if and only if $(G \otimes _R \mathbf{D}) ^{+} : 0  \to (G \otimes_R X )^{+}  \to  (G \otimes_R G_0)^{+}  \to (G   \otimes_R G_1)^{+}   \to \cdots  $ is acyclic. But the latter is naturally isomorphic to $\Hom_R  (\mathbf{D}, G  ^{+} ) : 0  \to \Hom _{R} (X,G^{+} )  \to  \Hom _{R} ( G_0, G^{+}  )   \to \Hom _{R} ( G_1, G^{+} )  \to \cdots  $ from the Proposition \ref{PropositionF-I} we have that the left $R$-module $G^{+ } \in \GI _{(\A, \Inj)} $, thus the last sequence is acyclic by  Lemma \ref{LemaBalance}. 
 \end{proof}
 
 \begin{cor} \label{ProjBalance}
Let $R$ be a ring, $(\Le, \A)$ a perfect or semi-complete duality pair in $\Modu (R)$ such that $\A ^{+} \subseteq \Le$ and that $  \A $ can be considered in $\Modu (R)$ and $\Le$ in $\Modu (R^{\op})$. Then $- \otimes _R -$ is left balanced over  $\GP _{(\Proj (R^{\op}), \Le)} ^{\gorro} \times \GP _{(\Proj (R) , \Le)} ^{\gorro}$ by $\GP _{(\Proj  (R^{\op}), \Le)}  \times \GP _{(\Proj , \Le)}$.
 \end{cor}
 
 \begin{proof}
 It follows from Theorem \ref{BalanF-P} and Proposition \ref{ProyectivosPlanos}.
 \end{proof}
 
 We are also interested in considering the balance conditions of the functor $- \otimes _R - $ over $\Modu (R^{\op}) \times \Modu(R)$ by  $\GF _{(\Flat (R^{\op}), \A)} \times \GF_{(\Flat, \A)}$. To do so, we will use some of the techniques employed in \cite{Yang}, where it is shown that for a Ding-Chen ring $R$ the functor  $- \otimes_R - $ is left balanced over $\Modu (R) \times \Modu(R)$ by   $\GF _{(\Flat, \Inj)} \times \GF_{(\Flat, \Inj)}$. Since the notion of $(\Le, \A)$-Gorenstein ring \cite{Wang20} is a generalization of AC-Gorenstein and Ding-Chen ring it is natural to ask whether on such rings the above-mentioned balance is given. We recall the notion of $(\Le, \A)$-Gorenstein ring.
 
 \begin{defi}\cite[Definition 4.2]{Wang20}
 Let $R$ a ring such that $(\Le, \A)$ is a bi-complete duality pair, and $m$ a nonnegative integer.
 \begin{enumerate}
 \item We say that $R$ is a \textbf{$(\Le, \A)$-Gorenstein ring} if $\resdim _{\A} (\Le) \leq m$. Equivalently, if $\coresdim_{\Le} (\A) \leq m$.
 \item $R$ is called \textbf{flat typed} $(\Le, \A)$-Gorenstein ring if $\resdim _{\A} (\Le), \fd (\A) \leq m$. 
 \end{enumerate}
 \end{defi}
 
 \begin{rk}
 Observe that if $R$ is a flat typed $(\Le, \A)$-Gorenstein ring with respect to  a bi-complete duality  pair  $(\Le, \A)$ and $\mathrm{l.w.Ggl} _{(\Flat(R), \A)}(R) < \infty$ then from \cite[Proposition 4.5]{Wang20} and Proposition \ref{EquivalenciaGlobal}, we have the equalities
 $$\mathrm{l.w.Ggl} _{(\Flat(R), \A)}(R) = \fd (\A) = \resdim _{\A} (\Le) = \coresdim_{\Le} (\A).$$
 \end{rk}
 
Now by   \cite[Thm 4.8]{Wang20} we know that in a $(\Le, \A)$-Gorenstein ring $R$ with respect to a bi-complete duality pair $(\Le, \A)$  the triple  
 
 $$( \GP _{(\Proj , \Le)} , \W , \GI _{(\A, \Inj)})$$ forms a  hereditary and complete cotorsion pair in $\Modu (R)$, this in particular tells us that $\GP _{(\Proj , \Le)} ^{\ortogonal}  = {^{\ortogonal} \GI _{(\A, \Inj)}}$. Furthermore, since $(\Le, \A)$  is in particular a symmetric duality pair, then the Proposition \ref{ProyectivosPlanos} implies that $\GP _{(\Proj , \Le)}  \subseteq \GF _{(\Flat, \A)}$ (see also \cite[Lemma 4.6 (1)]{Wang20}) which gives the containment 
 $$\GF _{(\Flat, \A)} ^{\ortogonal } \subseteq \GP _{(\Proj , \Le)}  ^{\ortogonal} = {^{\ortogonal} \GI _{(\A, \Inj)}}.$$

On the other hand,  J. Gillespie and  Alina I. \cite[Proposition 4.2]{Gill21} have recently proven that if $\GF _{(\Flat, \A)}$ is closed by extensions and $ \Inj (R^{\op}) \subseteq \A$ then
 $$(\GF _{(\Flat, \A)}, \GF _{(\Flat, \A)} ^{\ortogonal})$$
  form an hereditary and perfect  cotorsion pair in $\Modu (R)$.  From the latter it can be concluded that all $M \in \Modu (R)$ possesses a $\GF _{(\Flat, \A)}$-cover, this is, there is an exact sequence  $0 \to K \to G \to M$, with  $G \in  \GF _{(\Flat, \A)}$ and $K \in \GF _{(\Flat, \A)} ^{\ortogonal}$. From all this we can establish the following result. 
  
  \begin{lem} \label{BalanBiperfect}
   Let $R$ be a  $(\Le, \A)$-Gorenstein ring with respect to a bi-complete duality pair  $(\Le, \A)$  and for a $R$-module $M$ take a left proper $\GF _{(\Flat, \A)}$-resolution $ \mathbf{D} : \cdots \to G_1 \to G_0 \to M \to 0$. Then $G \otimes _R \mathbf{D}$ is acyclic for all  $R^{\op}$-module $G \in \GF _{(\Flat (R^{\op}), \A)}$.
  \end{lem}
  
  \begin{proof}
  From \cite[Corollary 5.3]{Gill21} the class $\GF_{(\Flat, \A)}$ is closed under extensions. Now, since $(\Le, \A)$ is in particular a perfect duality pair then $\Inj (R^{\op}) \subseteq \A$ \cite[Theorem 2.3]{Gill19} and $\GF _{(\Flat, \A)}$ is closed under extensions, from what has been said above, is enough to consider the short exact sequence  $0 \to K \to G \to M \to 0$ with $K \in  \GF _{(\Flat , \A)} ^{\ortogonal} $ and $G \in  \GF _{(\Flat , \A)}$, and prove that for all  $G \in  \GF _{(\Flat , \A)}$ the sequence $0 \to G \otimes K \to G \otimes_R  G_0 \to G \otimes _R  M \to 0$ is exact. Note that such sequence is exact if and only if the sequence 
  $$0 \to (G \otimes _R M) ^{+ } \to (G \otimes _R G_0) ^{+ } \to (G \otimes _R K) ^{+ } \to 0 $$ is exact. But the latter is naturally isomorphic to the sequence 
  $$0 \to \Hom _R (M, G^{+ }) \to \Hom _R ( G_0  , G^{+ } )  \to \Hom _R (K, G^{+ } ) \to 0.$$ 
  Since  $G \in \GF _{(\Flat (R^{\op}) , \A)}$ is a right $R$-module, then  by Proposition \ref{PropositionF-I} we have that  $G^{+} \in \GI _{( \A, \Inj)}$ is a left $R$-module. Thus there is an exact sequence $0 \to \tilde{G} \to I \to G^{+} \to 0$ with $I \in \Inj (R)$ y $\tilde{G} \in  \GI _{( \A, \Inj)}$. Consider the following exact and commutative diagram
  
  $$\xymatrix{  & \Hom _R  (G_0, \tilde{G})\ar@{^{(}->}[d]  & \Hom _R  (K, \tilde{G}) \ar@{^{(}->}[d]& \\
  \Hom _R  (M, I)  \ar@{^{(}->}[r] & \Hom _R  (G_0, I) \ar[r]\ar[d] & \Hom _R  (K, I)  \ar[r]\ar[d] & \Ext ^1_R (M, I)  =0\\
  \Hom _R  (M, G^{+ }) \ar@{^{(}->}[r] & \Hom _R  (G_0, G^{+ })  \ar[r]  & \Hom _R  (K, G^{+ }) \ar[d]
& \\
   &   & \Ext ^1 _R (K, \tilde{G})  &
}$$

 Thus, to prove that $\Hom _R ( G_0  , G^{+ } )  \to \Hom _R (K, G^{+} )$  is an epimorphism, it is enough to see that $\Ext ^1 _R (K, \tilde{G}) =0$. To this end, let us remember that $K \in  \GF _{(\Flat , \A)} ^{\ortogonal}$,  $\tilde{G} \in  \GI _{( \A, \Inj)}$ and note that $\GF _{(\Flat, \A)} ^{\ortogonal } \subseteq \GP _{(\Proj , \Le)}  ^{\ortogonal} = {^{\ortogonal} \GI _{(\A, \Inj)}}$. 
  \end{proof}
 
We can declare the following results.
  
  \begin{pro} \label{BalanWhole}
   Let $R$ be a  $(\Le, \A)$-Gorenstein ring with respect to a bi-complete duality pair  $(\Le, \A)$. Then the functor $-\otimes _R -$ is left balanced over $\Modu (R ^{\op}) \times \Modu (R)$ by $\GF _{(\Flat (R^{\op}), \A)} \times \GF _{(\Flat, \A)}$. 
  \end{pro}
  \begin{proof}
  Follows from Lemma \ref{BalanBiperfect} applied for $\GF _{(\Flat , \A)}(R)$ and $\GF _{(\Flat  (R^{\op}), \A)}(R^{\op})$.
  \end{proof}

  \begin{pro}
  Let $R$ be a coherent ring (on both sides) with $l.\mathrm{FPD} (R) < \infty$,  $r.\mathrm{FPD} (R) < \infty$ and $(\Le, \A)$ a semi-complete duality pair in $\Modu (R)$ and such that $(_{R^{\op}} \Le, \A _{R^{\op}})$ also  semi-complete in  $\Modu (R^{\op})$, then the functor $- \otimes _R -$ is left balanced over  $ \GF _{(\Flat (R^{\op}), \A)} ^{\gorro} \times \GF _{(\Flat, \A)} ^{\gorro}$ by $ \GF _{(\Flat (R^{\op}), \A)} \times \GF _{(\Flat, \A)}$.
  \end{pro}
  
  \begin{proof}
 From \cite[Corollary 5.3]{Gill21} the classes $\GF _{(\Flat, \A)}$ and $\GF _{(\Flat (R^{\op}), \A)}$ are closed under extensions. Now from Theorem \ref{AB4} (b) there exist a left proper $\GF _{(\Flat, \A)} $-resolution for each $M \in \GF _{(\Flat, \A)}  ^{\gorro}$.  Observe that in the proof of Lemma \ref{LemaBalance} has been showed that $\Ext ^1 _{R} (\Proj (R) ^{\gorro} , \GI _{(\A, \Inj)}) =0$, thus the proof in  Lemma \ref{BalanBiperfect} can be adapted by observing that  $\Ext _R ^1 (K, \tilde{G}) =0$ since by \cite[Proposition 6]{Jensen} we have that $K \in \Flat (R) ^{\gorro} \subseteq \Proj (R) ^{\gorro}$.
  \end{proof}
 
\section{Applications} \label{S6}
We will prensent here some consequences of the theory developed above. Although several of the preceding results can be rewritten in terms of the following duality pairs, we have chosen only a few of those we believe to be most significant. We begin remember some notions, and next give examples of duality pairs some of which have been considered in \cite{HJ,Wang20}.\\

Let $R$ be an arbitrary ring and $n \in \mathbb{N}$.
\begin{enumerate}
\item Recall that an $R$-module $M$ is called \textbf{FP-injective} \cite{Gill10} if $\Ext ^{1} _{R} (N,M) =0$ for all finitely presented $R$-modules $N$. Denote by $\mathcal{FP}Inj (R)$ the class of all FP-injective $R$-modules, and by $\Flat (R)$ the class of all flat $R$-modules. The pair $(\Proj (R), \Flat (R))$ is GP-admissible, the pair  $(\mathcal{FP}Inj (R), \Inj (R))$ is GI-admissible, and the classes $\GP _{(\Proj (R), \mathrm{Flat} (R))}$ and $\GI_{(\mathcal{FP}Inj (R)(R), \Inj (R))}$ are called \textbf{Ding projective modules} and \textbf{Ding injective modules},  respectively (also denoted by $\DP(R)$ and $\DI(R)$, respectively).
\item An $R$-module $F$ is of \textbf{type FP$_{\infty}$}  \cite{BGH} if it has a projective resolution $\cdots \to P_1 \to P_0 \to F \to 0$, where each $P_i$ is finitely generated.  An $R$-module $A$ is \textbf{absolutely clean} if $\Ext ^{1} _{R} (F,A) =0$ for all $R$-modules $F$ of type FP$_{\infty}$. Also an $R$-module $L$ is \textbf{level} if $\mathrm{Tor} _{1} ^{R} (F,L) =0$ for all (right) $R$-module $F$ of type FP$_{\infty}$. We denote by $\mathrm{AC} (R)$  the class  of all absolutely clean $R$-modules, and by  $\mathrm{Lev} (R)$ the class of all level $R$-modules. The pair $(\Proj (R), \mathrm{Lev} (R))$ is GP-admissible and the pair  $(\mathrm{AC}(R), \Inj (R))$ is GI-admissible. The classes of $R$-modules $\GP _{(\Proj (R), \mathrm{Lev} (R))}$ and $\GI_{(\mathrm{AC}(R), \Inj (R))}$ are called \textbf{Gorenstein AC-projective} and \textbf{Gorenstein AC-injective}, respectively.  
\item $M \in \Modu (R)$ is called \textbf{finitely $n$-presented} if $M$ possesses a projective resolution $0 \to P_n \to \cdots \to P_1 \to P_0 \to M \to 0$ for which each $P_i$ is finitely generated free. We denote by $\mathcal{FP}_n(R)$ the class of all finite $n$-presented $R$-modules. A ring $R$ is said \textbf{left $n$-coherent} if $\mathcal{FP}_n(R) \subseteq \mathcal{FP}_{n+1}(R)$.
\item An $R$-module $E$ is called \textbf{$\mathcal{FP}_n$-injective} if for all $M \in \mathcal{FP}_n(R)$, $\Ext _R ^1 (M, E) =0$. Also an $R^{op}$-module $H$ is called \textbf{$\mathcal{FP}_n$-flat} if for all $M \in \mathcal{FP}_n(R)$, $\Tor _1 ^R (H, M) =0$.
\item Given an $R$-module $M$, is said that $M$ is  (resp. strongly) \textbf{$w$-$FI$-flat} if $\Tor _1 ^R (E, M) =0$ (resp. $\Tor _i ^R (E, M) =0$, for all $i>0$) for any absolutely $w$-pure $R^{\op}$-module $E$ (for details see \cite{Tamek}). Also is said that $M$ is (resp. strongly) \textbf{$w$-$FI$-injective} if $\Ext^{1} _R (E,M) =0$ (resp. $\Ext^{i} _R (E,M) =0$ for all $i>0$) for any $w$-absolutely pure $R$-module $E$. In a similar way is also possible to define (strongly) \textbf{$w$-$FI$-projective} $R$-modules.
\end{enumerate}

We recall examples of perfect and complete duality pairs.
\begin{ex} \label{Complete}
Let $R$ be a ring.
\begin{enumerate}
\item Take ($n \geq 2$), then by \cite[Theorems 5.5 and 5.6]{BP} the pair 
$$( \mathcal{FP}_n\mbox{-}Flat (R),\mathcal{FP}_n\mbox{-}Inj (R^{\op}))$$ 
of  $\mathcal{FP}_n$-flat  $R$-modules and $\mathcal{FP}_n$-injective $R^{\op}$-modules form a duality pair in $\Modu (R)$. And from \cite[Corollarry 3.7]{BGH} is complete duality pair for any ring $R$.

\item Let $R$ commutative and Noetherian with a semidualizing $R$-complex $C$. There is associated two associated classes called the Auslander class, denoted $\A _0 ^{C}$, and the Bass class, denoted $\B _0 ^C$. Then by
\cite[Proposition 2.4]{HJ} the pair $(\A_0 ^{C} , \B _{0} ^C)$ is a complete duality pair.
\item By \cite[Remark 1, Proposition 6]{Tamek} we have that the pair  of (strongly) $w$-$FI$-flat $R$-modules and (strongly) $w$-$FI$-injective $R^{\op}$-modules $$(w\mbox{-}FI\mbox{-}Flat (R), w\mbox{-}FI\mbox{-}Inj (R^{\op}))$$  is a perfect duality pair (resp. the pair  $(\mathcal{S}w\mbox{-}FI\mbox{-}Flat (R), \mathcal{S}w\mbox{-}FI\mbox{-}Inj (R^{\op}))$ is a perfect duality pair).
 \end{enumerate}
\end{ex}

From the Theorem \ref{L-ABalan} and the previous list of complete and perfect duality pairs, we have the following classes of $R$-modules that makes the $\Hom _R (-,-)$ right balanced.
\begin{itemize}
\item[$\bullet$]  For any ring $R$, the functor $\Hom _{R} (-,-)$ is right balanced over
 $$\GP _{(\Proj , \mathcal{FP}_n\mbox{-}Flat (R))} ^{\gorro} \times \GI _{(\mathcal{FP}_n\mbox{-}Inj (R), \Inj)} ^{\cogorro},$$ 
 by $\GP _{(\Proj , \mathcal{FP}_n\mbox{-}Flat (R))} \times \GI _{(\mathcal{FP}_n\mbox{-}Inj (R), \Inj)}.$
\item[$\bullet$]  For a  commutative and Noetherian ring $R$ with a semidualizing $R$-complex $C$, the functor $\Hom _{R} (-,-)$ is right balanced over $\GP _{(\Proj , \A_0 ^{C})} ^{\gorro} \times \GI _{(\B_0 ^{C}, \Inj)} ^{\cogorro}$ by $\GP _{(\Proj , \A_0 ^{C} )} \times \GI _{(\B_0 ^{C}, \Inj)}$
\item[$\bullet$] For any ring $R$, the functor $\Hom _{R} (-,-)$ is right balanced over 
$$\GP _{(\Proj , w\mbox{-}FI\mbox{-}Flat (R))} ^{\gorro} \times \GI _{(w\mbox{-}FI\mbox{-}Inj (R), \Inj)} ^{\cogorro},$$
 by $\GP _{(\Proj , w\mbox{-}FI\mbox{-}Flat (R) )} \times \GI _{(w\mbox{-}FI\mbox{-}Inj (R), \Inj)}$. Furthermore is right balanced over 
 $$\GP _{(\Proj , \mathcal{S}w\mbox{-}FI\mbox{-}Flat (R))} ^{\gorro} \times \GI _{(\mathcal{S}w\mbox{-}FI\mbox{-}Inj (R), \Inj)} ^{\cogorro},$$
 by $\GP _{(\Proj , \mathcal{S}w\mbox{-}FI\mbox{-}Flat (R) )} \times \GI _{(\mathcal{S}w\mbox{-}FI\mbox{-}Inj (R), \Inj)}$.
\end{itemize}

Now from Corollary \ref{ProjBalance} and Example \ref{Complete} we have the following classes of $R$-modules that makes the $-\otimes _R - $ left balanced.

\begin{itemize}
\item[$\bullet$]  For any ring $R$, the functor $-\otimes _R - $ left balanced over   $$\GP _{(\Proj (R^{\op}) , \mathcal{FP}_n\mbox{-}Flat (R^{\op}))} ^{\gorro} \times \GP _{(\Proj (R) , \mathcal{FP}_n\mbox{-}Flat (R))} ^{\gorro},$$ by $\GP _{(\Proj (R^{\op}) , \mathcal{FP}_n\mbox{-}Flat (R^{\op}))} \times \GP _{(\Proj (R) , \mathcal{FP}_n\mbox{-}Flat (R))} .$
\item[$\bullet$] For a  commutative and Noetherian ring $R$ with a semidualizing $R$-complex $C$ the functor $-\otimes _R - $ left balanced over 
$$\GP _{(\Proj , \A_0 ^{C})} ^{\gorro} \times \GP _{(\Proj , \A_0 ^{C})} ^{\gorro},$$
 by $\GP _{(\Proj , \A_0 ^{C})} \times \GP _{(\Proj , \A_0 ^{C})}$.
\end{itemize}

The following are examples of bi-complete duality pairs.

\begin{ex} \label{E-Bi-complete}
Let $R$ be a ring.
\begin{enumerate}
\item The pair $(\mathrm{Lev} (R), \mathrm{AC} (R^{\op}))$ forms a duality pair in $\Modu (R)$, since from \cite[Proposition 2.7 and 2.10]{BGH} $M \in \mathrm{Lev} (R)$ if and only if $M^{+} \in \mathrm{AC} (R^{\op})$, and both classes are closed under direct summands. Furthermore $\mathrm{Lev} (R)$ is projective resolving and closed under arbitrary coproducts and from \cite[Example 3.6]{Wang20} is a bi-complete duality pair.

\item From \cite[Theorem 2.1]{Fhouse} the pair $(\Flat (R), \mathcal{FP}Inj (R^{\op}))$ form a duality pair in $\Modu (R)$. If $R$ is a right coherent ring then from \cite[Theorem 2.2]{Fhouse} $(\mathcal{FP}Inj (R^{\op}) , \Flat (R))$ is a duality pair in $\Modu (R^{\op})$. Furthermore from \cite[Example 3.8 (1)]{Wang20} $(\Flat (R), \mathcal{FP}Inj (R^{\op}))$ is a bi-complete duality pair when $R$ is right coherent.

\item Now suppose that $R$ is  right $n$-coherent ($n \geq 2$), then from  \cite[Example 3.8 (2)]{Wang20} the pair $( \mathcal{FP}_n\mbox{-}Flat (R),\mathcal{FP}_n\mbox{-}Inj (R^{\op}))$ of  $\mathcal{FP}_n$-flat  $R$-modules and $\mathcal{FP}_n$-injective $R^{\op}$-modules is a bi-complete duality pair    when $R$ is right $n$-coherent.
 \end{enumerate}
\end{ex}

From Proposition \ref{BalanWhole} and Example \ref{E-Bi-complete} we have the following balance situations for $-\otimes _R -$, in the whole category $\Modu (R)$.
\begin{itemize}
\item[$\bullet$] For an \textit{AC-Gorenstein ring} $R$ \cite[Definition 4.1]{Gill18}  the functor $-\otimes _R - $ is left balanced over $\Modu (R^{\op}) \times \Modu (R)$ by  
$$ \GF _{(\Flat (R^{\op}), \mathrm{AC} (R^{\op}) )} \times \GF _{(\Flat , \mathrm{AC} (R))} $$
since from \cite[Example 4.3 (1)]{Wang20} the \textit{AC-Gorenstein rings} coincide with the  $(\mathrm{Lev (R), \mathrm{AC} (R^{\op})})$-Gorenstein rings.
\item[$\bullet$] For a right coherent and $(\Flat (R), \mathcal{FP}Inj (R^{\op}))$-Gorenstein ring $R$ the functor $-\otimes _R - $ is left balanced over $\Modu (R^{\op}) \times \Modu (R)$ by 
$$ \GF _{(\Flat (R^{\op}), \mathcal{FP}Inj (R^{\op}) )} \times \GF _{(\Flat , \mathcal{FP}Inj (R))}, $$
note that by \cite[Example 4.3 (3)]{Wang20} the $(\Flat (R), \mathcal{FP}Inj (R^{\op}))$-Gorenstein rings are exactly the Ding-Cheng rings. Thus this example recovers the result given in \cite[Theorem 3.23]{Yang}, since by \cite[Proposition 3.13]{Gill10} 
$$\GF _{(\Flat , \mathcal{FP}Inj (R))} = \GF _{(\Flat (R), \Inj (R)} .$$

\item[$\bullet$] For a Gorenstein $n$-coherent  ring $R$ ($n\geq 2$)  the functor $-\otimes _R - $ is left balanced over $\Modu (R^{\op}) \times \Modu (R)$ by 
$$ \GF _{(\Flat (R^{\op}), \mathcal{FP}_n\mbox{-}Inj (R^{\op}) )} \times \GF _{(\Flat , \mathcal{FP}_n\mbox{-}Inj (R))}, $$
this is true since by \cite[Example 4.3 (4)]{Wang20} the Gorenstein $n$-coherent  rings are exactly the flat-typed $( \mathcal{FP}_n\mbox{-}Flat (R),\mathcal{FP}_n\mbox{-}Inj (R^{\op})) $-Gorenstein rings.
\end{itemize}
We conclude this work by considering the situation when the Ding-projective $R$-modules  $\DP (R)$ coincide with the Gorenstein projective $R$-modules $\GP (R)$. By definition and without conditions on the ring $R$ it is known that $\DP (R) \subseteq \GP (R)$. Recently A. Iacob \cite{Alina20} has considered this question as part  of the tools developed  around the study of class $\GF _{(\Proj , \mathcal{B})} (R)$ where some occasions $\mathcal{B}$ corresponds to a  \textit{definable class}. It is known that over a Ding-Chen ring $R$ the class of Ding-projective $R$-modules coincides with the class of Gorenstein projective $R$-modules \cite[Theorem 1.1]{Gill17}, here we will see that the above also occurs when the finitistic projective dimension of $R$ in finite.

\begin{lem} \label{Ortogonalidad}
Let $R$ be a ring with $ \mathrm{LeftFPD} (R) < \infty$ and  $M \in \Modu (R)$. If for all  $m > 0 $, $\Ext  ^{m} _{R}(M,\Proj (R)) =0$, then   $\Ext ^m _R (M,\Flat (R) ^{\gorro}) =0$ for all $m> 0$.
\end{lem}

\begin{proof}
Consider $F \in \Flat (R)$, then by  \cite[Proposition 6]{Jensen} we have that $\pd (F) < \infty$, thus $\Flat (R) \subseteq \Proj(R) ^{\gorro}$. From this, it follows that $\Flat (R) ^{\gorro} \subseteq \Proj (R) ^{\gorro}$.  Now for $T \in \Flat (R) ^{\gorro}$ there is a projective resolution
$$ 0 \to Q_n \to \cdots \to Q_1 \to Q_0 \to T \to 0,$$
since $\Ext  ^{m} _{R}(M,Q) =0$ for all $m > 0 $ and all $Q \in \Proj (R)$, by a dimension shifting we have $\Ext ^{j} _R (M, T) \cong \Ext _R ^{j+n} (M,Q_n) =0 $, where the last term is zero since $Q_n \in \Proj (R)$. Thus, $\Ext _R ^j (M, T) =0$ for all $j>0$. 
\end{proof}

\begin{cor}\label{Finitista}
Let $R$ be a ring such that $ \mathrm{LeftFPD} (R) < \infty$, then $ \DP (R) = \GP (R)$.
\end{cor}
 
 \begin{proof}
Indeed, by definition and without conditions on the ring $\DP (R) \subseteq \GP (R)$. Now let's consider an exact complex of projectives $\mathbf{P}$, such that the complex  $\Hom _{R} (\mathbf{P}, \Proj (R))$ is an acyclic complex, by using  \cite[Lemma 3.10]{BMS} we have that all the cycles of the complex $Z^{i} _{\mathbf{P}} \in {^{\ortogonal} \Proj(R)}$, accordingly by the Lemma \ref{Ortogonalidad} we have that $Z^{i} _{\mathbf{P}} \in {^{\ortogonal} \Flat (R)} $, and again using \cite[Lemma 3.10]{BMS}, we obtain that the complex $\Hom _{R} ( \mathbf{P}, \Flat (R))$ is acyclic, thus we obtain the containment $\GP (R) \subseteq \DP (R)$.
  \end{proof}

The following result shows how to know when equality $\GP (R) = \DP(R)$ occurs,  namely, when $\id (\Flat (R)) < \infty$.

\begin{pro}\label{IgualdadDing}
Let $R$ be a ring such that  $\Flat (R) \subseteq \Inj (R) ^{\cogorro}$, then $\GP (R) = \DP (R)$.
\end{pro}

\begin{proof}
Consider an exact complex $\mathbf{P}$ of projective $R$-modules  such that the complex  $\Hom _R (\mathbf{P}, \Proj (R))$ is acyclic. From the proof of  \cite[Lemma 3.6]{Becerril21} we see that for each $F \in \Flat (R) \subseteq \Inj (R) ^{\cogorro}$ we have that  $\Hom _{R} (\mathbf{P}, F)$ is an acyclic complex. Thus we obtain the containment 
$\GP (R) \subseteq \DP (R)$.
\end{proof}

\begin{rk}
We can see from \cite[Proposition 3.10]{Becerril21} and the previous result, that for a ring $R$ if  $\mathrm{gl.GP}(R) < \infty$ and $\Flat (R) \subseteq \Inj (R) ^{\cogorro}$, then $\id (\Flat (R)) < \infty$.
\end{rk}

  Now we sill see that the Corollary \ref{Finitista} is more general than how it was stated. As we see below, it can be stated in terms of GP-admissible pairs in the following result.
 
  \begin{pro} \label{Aplastado}
  Consider a GP-admissible pair $(\X ,\Y)$ in an abelian category $\C$ with $\omega := \X \cap \Y$, such that $\Y ^{\gorro}$, $\omega$ and $\X \cap \Y ^{\gorro}$ are closed under direct summands in $\C$. If $(\X, \overline{ \Y}) \subseteq \C^{2}$ is a pair with $\Y \subseteq \overline{\Y} \subseteq \Y ^{\gorro}$, then equality  $\GP _{(\X, \Y)} = \GP _{(\X, \overline{\Y})}$ is given.
  \end{pro}
\begin{proof}
 From the containment $\Y \subseteq \overline{\Y} \subseteq \Y ^{\gorro}$ we have that 
$$\GP _{(\X, \Y^{\gorro})} \subseteq \GP _{(\X, \overline{\Y})} \subseteq \GP _{(\X, \Y)},$$ 
while by \cite[Theorem 3.34 (d)]{BMS} we have that $\GP _{(\X, \Y^{\gorro})}  = \GP _{(\X, \Y)} $, from which we obtain the desired equality.
\end{proof}

 It is also naturally interesting to know when the class of the AC-Gorenstein projectives $\GP _{AC} (R)$ coincides with the class $\GP (R)$, we establish this as follows.
 
 \begin{cor}
Let us consider the class $\GP _{\mathrm{AC}}(R)$ of all the AC-Gorenstein projective $R$-modules. If any of the conditions $\mathrm{Lev} (R) \subseteq \Proj (R) ^{\gorro}$ or  $\mathrm{Lev} (R) \subseteq \Inj (R) ^{\cogorro}$, then equality $\GP (R) = \GP _{\mathrm{AC}}(R)$ is given.
 \end{cor} 
 
 \begin{proof}
  It follows from the Proposition \ref{Aplastado}  or with a  similar proof to the one given in the Proposici\'on \ref{IgualdadDing}, respectively.
 \end{proof}

\textbf{Acknowledgements} The author thanks to professor Marco A. Per\'ez for suggestions and nice discussions about  this paper. The author was fully supported by a  CONACyT Postdoctoral Fe\-llow\-ship.

\end{document}